\documentclass{article}
\usepackage{amsmath, amsthm}
\usepackage{amssymb}
\usepackage{microtype}
\usepackage{todonotes}
\usepackage[applemac]{inputenc}
\usepackage{graphicx}
\usepackage{color}
\usepackage{tikz}
\usepackage{enumitem}
\setlist[enumerate,1]{label=(\roman*)}

\usetikzlibrary{matrix}
\def\F{{\mathcal F}}

\def\I{{\cal I}}

\def\L{{\mathcal L}}

\def\NN{\mathbb{ N}}

\def\P{{\mathcal P}}

\def\R{{\mathbb R}}

\def\X{{\mathcal X}}

\def\Y{{\mathcal Y}}

\def\PROB{{\mathbb P}}

\def\I\nd{{\mathbb I}}

\newtheorem{theorem}{Theorem}[section]

\newtheorem{corollary}[theorem]{Corollary}

\newtheorem{lemma}[theorem]{Lemma}

\newtheorem{proposition}[theorem]{Proposition}

\newtheorem{definition}[theorem]{Definition}

\newtheorem{remark}[theorem]{Remark}
\newtheorem{example}[theorem]{Example}

\renewcommand{\epsilon}{\varepsilon}
\newcommand{\nd}{\text{nd}}

\numberwithin{equation}{section}
\title{ \bf Fundamental Properties of Process Distances
}
\author{
Julio Backhoff Veraguas
\and     Mathias Beiglb\"ock\thanks{J.\ Backhoff, M.\ Beiglb\"ock, and M.\ Eder acknowledge  support by the Austrian Science Fund (FWF) under grant Y782-N25.}
\and Manu Eder %\thanks{text} 
\and Alois Pichler 
 }
 
\begin{document}

\maketitle

\begin{abstract}
%Official, 100 word abstract for SPA: Information is an inherent component of stochastic processes and to measure the distance between them it is not sufficient to consider the distance between their laws. 
%The bicausal Wasserstein distance addresses this challenge by incorporating the filtration. It is of emerging importance due to its applications in stochastic analysis, stochastic programming, mathematical economics and related disciplines. 
%	This article establishes several fundamental properties of the bicausal Wasserstein distance. In particular we prove that it generates a Polish topology but is itself not a complete metric. We identify its completion to be the set of nested distributions recently introduced by Pflug.
	
	Information is an inherent component of stochastic processes and to measure the distance between different stochastic processes it is not sufficient to consider the distance between their laws. Instead, the information which accumulates over time and which is mathematically encoded by filtrations has to be accounted for as well.
	The \emph{nested distance/bicausal Wasserstein distance} addresses this challenge by incorporating the filtration.
	It is of emerging importance due to its applications in stochastic analysis, stochastic programming, mathematical economics and other disciplines.

	This article establishes a number of fundamental properties of the nested distance. In particular we prove that the nested distance of processes generates a Polish topology  but is itself not a complete metric. We identify its completion to be the set of \emph{nested distributions}, which are a form of generalized stochastic processes.  We also characterize the extreme points of the set of couplings which participate in the definition of the nested distance, proving that they can be identified with adapted deterministic maps. Finally, we compare the nested distance to an alternative metric, which could possibly be easier to compute in practical situations.  
	
\end{abstract} 


\section{Introduction}\label{Setting}

In this paper we consider several distances between  stochastic processes and investigate  their fundamental metric and topological properties. The distances we discuss are based on transport theory; we refer to Villani's monographs~\cite{Vi03,Vi09} or the lecture notes by Ambrosio and Gigli~\cite{AmGi13} for background. %, cf.\ \cite{Villani}. 
 Classical transport distances (cf.\ \cite{RaRu98,RaRu98II}) %, such as the Wassertein ones, 
do not respect the information structure inherent to a multivariate distribution when this is seen as a stochastic process. It is therefore desirable to find  natural extensions of these distances that do take information/\,filtrations into account. To achieve this, one has to  adjust the definition of transport distances and include constraints involving the filtration to incorporate information  at specified times. Heuristically this means that the computation of the distance is done over transport plans/\,couplings that only move mass \emph{respecting} the \emph{causal} structure inherent to filtrations.
 
These ideas lead to the \emph{nested distance} introduced by Pflug~\cite{Pflug2009}, which is computed just like the Wassertein distance but where only transport plans/couplings respecting the given filtration are considered. Its systematic investigation was continued in \cite{PflugPichler2011, PflugPichlerBuch,PflugPichlergeneration}. %Independly 
  The nested distance has already turned out to be a crucial tool for applications in the field of multistage stochastic optimization, where problems can be computationally extremely challenging and in many situations they simply cannot be managed in reasonable time. Based on the nested distance, approximations with tractable simplifications become feasible and sharp bounds for the approximation error can be found. 
%
%  The nested distance derives its importance from multistage stochastic optimization, an emerging discipline in applied mathematics. Multistage stochastic optimization handles optimization problems under uncertainty by implementing and executing a policy of consecutive managerial actions. 
%Multistage stochastic optimization problems are computationally extremely challenging, in many situations they simply cannot be managed in reasonable time. For this reason it is necessary to have simplifications at hand with reduced and manageable computational effort. The nested distance provides a tight upper bound for these problems, provided that the objective function (the cost or profit function) is Lipschitz continuous. More generally, transport distances which take filtrations into account can be expected to provide information and bounds on Multistage stochastic optimization problems.
%
Independently, a systematic treatment and use of  causality as an interesting property of abstract transport plans and their associated optimal transport problems was initiated by  Lassalle in \cite{Lassalle}; in particular he introduces a nested distance under the name \emph{(bi)causal Wasserstein distance} and provides intriguing connections with classical geometric/\ functional analytic inequalities, as well as stochastic analysis. %Continuing this line of research, 
In \cite{{BBLZ}} it is argued that the renowned Knothe--Rosenblatt rearrangement (cf.\ \cite[Section 1.4]{Vi09}), also known as quantile transform in statistics, is a causal analogue to the celebrated Brenier-mapping in optimal transport. In \cite{AcBaZa16} it is established  that causality is naturally linked to the subject of enlargement of filtrations in stochastic analysis, and that continuous-time analogues to the nested distance can be used to provide robust bounds for optiml stopping problems. These articles are in the wider tradition of constrained transport problems and in particular related to martingale optimal transport (cf.~\cite{BHP,GHT,DoSo12, BeCoHu14, NuSt16} among many others).

While the nested distance has received  attention in different fields, a number of basic and fundamental questions on the corresponding  topological/\,metric structure were widely open. 
 The main goal of the present article  is to fill this gap. 
%One prominent example, which we recall in \eqref{nd} and analyse in detail in this article, is the \emph{nested distance} (also called \emph{process} or \emph{multistage distance}), introduced by Pflug in \cite{Pflug2009}. %It is bicausal and accounts equally and symmetrically for the information of both processes involved.
%Lipschitz continuity is a key ingredient of the Kantorovich--Rubinstein theorem, and this is a similarity to the classical transportation or Wasserstein distance. However, 
Although the nested distance is inspired by the Wasserstein distance, it turns out that there are substantial differences between these concepts. In Section~\ref{sec:Nested} we show that convergence in nested distance cannot be verified by testing against a class of usual functions. Furthermore we observe that nested distance is an  incomplete metric.
Our first main result, Theorem \ref{thm: embedding}, identifies the completion of this topology explicitly. 
This completion turns out to be the space of \emph{nested distributions}, which are in a sense generalized stochastic processes,  equipped with a classical Wasserstein distance (cf.\ \cite{Pflug2009} and Definition \ref{def:Nested} below). We thus connect two hitherto unrelated mathematical objects in an unexpected way.  In Section~\ref{sec:Weak} we establish, on the other hand, that the topology induced by the nested distance is Polish, i.e., separable and completely metrizable (cf.\ Theorem~\ref{thm polish} and comments thereafter). This is the second main result of the article. As a consequence of these considerations we can moreover find a complete metric compatible with the nested distance.

Most naturally, the nested distance is defined as a variant of a transport distance where the cost function is based on the usual Euclidean distance on $\R^N$. We want to emphasize, however, that other metrics on the underlying space $\R^N$ are often more convenient in applications, but this does not affect the topology. In the classical setup, a bounded metric on $\R^N $ induces a Wasserstein distance corresponding to the weak topology on probability measures. Likewise  a nested distance based on a bounded metric induces an information-compatible topology which might be seen as  a \emph{weak nested topology}.
%
%in reference to the Euclidean distance on $\R^N$. 
%Having classical transport theory in mind, the mo
%
%The nested distance is not the only information-compatible metric on processes that we look at. We shall also define a topology, which we call \emph{weak nested topology} and which is in a sense is closest to the weak* topology on measures, while still better taking into account the filtration structure. For its proximity with the weak* topology, we are led to study this topology and eventually establish its Polish character in what is our second main result, Theorem \ref{thm polish}. This result easily extends to the topology of nested distance too, and we envision that this could have future applications. By having identified the completion and by proving the Polish character of the nested distance topology, we can say that we have much advanced the understanding of this important process distance. One open question that is of potential interest, is the characterization of relative compactness in either the weak nested topology or in nested distance. The third information-compatible metric we study in this article is related to the Knothe--Rosenblatt rearrangement (also called quantile transform), so we dub it \emph{KR distance}. Despite this rearrangement being of great use all over mathematics, our finding is that KR distance is strictly stronger than the nested distance. This suggests that KR distance is unsuitable for both applications and theoretical purposes. 
%{\color{blue}
We establish that this weak nested topology  coincides with the \emph{information topology} introduced by Hellwig  \cite{Hellwig} (cf.\ also~\cite{Barbie} and the references therein) for applications in the field of mathematical economics,  more specifically, sequential decision making and equilibria. %The author defines a concept related to our weak nested topology, see Section \ref{comparison Hellwig} below, and establishes that certain stochastic optimization problems are continuous w.r.t.\ the noise distribution by using his topology. We show in fact that the author's concept is equivalent to our weak nested topology.} 
%This means that our  
%
%It remains to be seen whether the author's concept coincides with ours; we only show that our notion is not any weaker. At any rate we see a chance that our work can be relevant to the mathematical economics community. Furthermore, our 
Our  results in Sections~\ref{sec:Completeness} and \ref{sec:Weak} %on completness as well as Theorem \ref{thm polish} 
seem to be novel in the setup of \cite{Hellwig} and potentially applicable in mathematical economics. %We stress
To the best of our knowledge, the concept of nested distance (in contrast to weak nested topology) is not present nor related to the works \cite{Hellwig,Barbie}. %This means that it could also be of use to the aforementioned community. 

For computational reasons any \emph{simplification} of a measure/law of a stochastic process, is of particular interest in stochastic optimization and applications. For instance, the Knothe--Rosenblatt rearrangement and Brenier-mapping provide simple/computable transformations of measures, and so may be used to simplify the starting measure. One reason deterministic (i.e.\ Monge) maps are so useful in applications, is that they are dense in the set of all couplings with a given initial marginal, and more importantly, coincide with the extreme point of such set of couplings. In Section~\ref{sec:extreme} we will establish the analogue result when filtrations are considered: deterministic and adapted transformations of an initial measure, correspond exactly to the extreme points of the set of all couplings preserving causality (i.e.\ filtrations) and having the same initial measure. % We believe that the results presented are of interest in real-life applications.

%As was already mentioned, the Knothe--Rosenblatt rearrangement is the causal analogue of Brenier's map. This rearrangement appears in various places in  mathematics, from statistics to the theory of geometric inequalities. For this reason we  also compare in Section~\ref{sec:KnotheRosenblatt} the nested distance to a new distance defined in terms of (i.e., induced by) this rearrangement, which a priori is easier to compute. We believe that this new distance could prove valuable in applications. We first notice that both distances actually coincide in dimension one, and it would be notable and of much interest, if they coincided more generally. Our finding here is that, in higher dimensions, this new distance is strictly stronger than the nested distance. This even leads us to conjecture that in multiple dimensions there is no privileged transport/rearrangement that may induce a \textit{simpler} metric topologically equivalent to the nested distance.

%We further provide two different settings of the nested distance, which are in some way equivalent, but very different from another perspective: the first setting is based on a usual probability space and is not complete, while we characterize its completion constructively by an abstract space.

\medskip

\paragraph{\bf Outline.} We introduce the notation used throughout the paper and describe the mathematical setup in Section~\ref{sec:Notation}. Section \ref{sec:Nested} discusses elementary properties of the nested distance.  Section~\ref{sec:Completeness} is concerned with completeness-properties of this distance. {Then, in Section~\ref{sec:Weak}, we introduce the weak nested topology, compare it to Hellwig's information topology and establish its Polish character. } Section~\ref{sec:extreme} discusses extreme points of important sets associated with the nested distance, while Section~\ref{sec:KnotheRosenblatt} introduces the Knothe--Rosenblatt distance and provides a comparison with the nested distance.  In Section~\ref{sec:Summary} we conclude with a brief summary.

\section{Notation and mathematical setup}\label{sec:Notation}

%\textbf{Notation.} 
The ambient set throughout this article is $\R^N$, which we consider as a filtered space endowed with the canonical (i.e., coordinate) filtration $\big(\F_t\big)_{t=1}^N$. (More precisely $\F_t$ is the smallest $\sigma$-algebra on $\R^N$ such that the projection $\R^N\ni x\mapsto (x_1,\dots,x_t)\in \R^t$ onto the first $t$ components is Borel-measurable, and so forth.) % {\color{blue} We endow $\R^N$ with a metric $d$ so that $\big(\mathbb R^N,d\big)$ is complete and separable. This is done for convenience of presentation and notation; for most results one may substitute $S^N$ for $\R^N$, where $S$ is a Polish space.}

We endow $\R^N$ with the $\ell^p$-type product metric
\begin{equation}
\label{eq:dp}
d(x,y)\,:=\,d_p(x,y)\,:= \,\sqrt[p]{\sum_{t=1}^N \underline{d}(x_i,y_i)^p},\qquad p\in [1,\infty),
\end{equation}
for some base metric $\underline{d}$ on $\R$ compatible with the usual topology. We are particularly interested in the cases where $\underline{d}$ is the usual distance or is a compatible bounded metric on $\R$. Notably, for most results one may replace $\R^N$ by $S^N$, where $S$ is a Polish space and endow $S^N$ with an $\ell^p$-type product metric again. Throughout this work we fix $\underline{d}$, $p$ and $d$ as described.

%This is done for convenience of presentation and notation; for most results one may substitute $S^N$ for $\R^N$, where $S$ is a Polish space. 
% we restrict the presentation to processes in finite time and to the real numbers $\mathbb R$ as the state space, if this is possible without loss of generality.

The pushforward of a measure $\gamma$ by a map $M$ is denoted by $M_*\gamma:=\gamma\circ M^{-1}$.
For a product of sets $\X \times \Y$ we denote by $p^1$ ($p^2$, resp.) the projection onto the first (second, resp.) coordinate. We denote by $\gamma^x$, $\gamma^y$ the regular kernels of a measure $\gamma$ on $\X \times \Y$ w.r.t.\ its first and second coordinate, respectively, obtained by disintegration\footnote{When writing $\gamma^x(B)$ and $\gamma^y(A)$ it should be understood $\gamma^x(\{x\}\times B)$ and $\gamma^y(A\times\{y\})$ respectively.} (cf.\ \cite{Ambrosi2005}) so that $\gamma(A\times B)=\int_A \gamma^{x_1}(B)\,\gamma^1\big(\mathrm d x_1\big)$ with $\gamma^1(A):=p^1_* \gamma(A)=\gamma(A\times\Y)$.
The notation extends analogously to products of more than two spaces. 
We convene that for a probability measure $\eta$ on $\R^N$, $\eta^{x_1,\dots,x_t}$ denotes the one-dimensional measure on $x_{t+1}$ obtained by disintegration of $\eta$ w.r.t.\ $(x_1,\dots,x_t)$.

A statement like ``for $\eta$-a.e.\ $x_1,\dots,x_t$'' is meant to denote ``almost-everywhere'' with respect to the projection of $\eta$ onto the coordinates $(x_1,\dots,x_t)$. On $\R^N\times \R^N$ we denote by $(x_1,\dots,x_N)$ the first half and by $(y_1,\dots,y_N)$ the second half of the coordinates. Similarly, we use the convention that for a probability measure $\gamma$ on $\R^N\times\R^N$, $\gamma^{x_1,\dots,x_t,y_1,\dots,y_t}$ denotes the two-dimensional measure on $(x_{t+1},y_{t+1})$ given by regular disintegration of $\gamma$ w.r.t.\ $(x_1,\dots,x_t,y_1,\dots,y_t)$, so a statement like ``for $\gamma$-a.e.\ $x_1,\dots,x_t,y_1,\dots,y_t$'' is meant to denote ``almost-everywhere'' with respect to the projection of $\gamma$ onto $x_1,\dots,x_t,y_1,\dots,y_t$. 

%For a product of sets $\X \times \Y$ we denote by $p^1$ ($p^2$, resp.) the projection onto the first (second, resp.) coordinate. The pushforward of a measure $\gamma$ by a map $M$ is $M_*\gamma:=\gamma\circ M^{-1}$.
%We denote by $\gamma^x$, $\gamma^y$ the regular kernels of a measure $\gamma$ on $\X \times \Y$ w.r.t.\ its first and second coordinate respectively obtained by disintegration (cf.\ \citet{Ambrosi2005}) so that $\gamma(A\times B)=\int_A \gamma^{x_1}(B)\,\gamma^1\big(\mathrm d x_1\big)$ with $\gamma^1(A):=p^1_* \gamma(A)=\gamma(A\times\Y)$.
%The notation extends analogously to products of more than two spaces.
%
%On $\R^N\times \R^N$ we denote by $(x_1,\dots,x_N)$ the first half and by $(y_1,\dots,y_N)$ the second half of the coordinates and we convene that for a probability measure $\gamma$ in $\R^N\times\R^N$ ($\eta$ on $\R^N$, resp.), $\gamma^{x_1,\dots,x_t,y_1,\dots,y_t}$ ($\eta^{x_1,\dots,x_t}$, resp.) denotes the two-dimensional measure on $(x_{t+1},y_{t+1})$ (one-dimensional measure on $x_{t+1}$, resp.) given by regular disintegration of $\gamma$ w.r.t.\ $(x_1,\dots,x_t,y_1,\dots,y_t)$ ($\eta$ w.r.t.\ $(x_1,\dots,x_t)$, resp.). Also, a statement like ``for $\gamma$-a.e.\ $x_1,\dots,x_t,y_1,\dots,y_t$'' or ``for $\eta$-a.e.\ $x_1,\dots,x_t$'' is meant to denote ``almost-everywhere'' with respect to the projections of $\gamma$ onto $x_1,\dots,x_t,y_1,\dots,y_t$ or $\eta$ onto $x_1,\dots,x_t$, respectively. 

The probability measures on the product space $\R^N \times \R^N$ with marginals $\mu$ and $\nu$ constitute the  possible \emph{transport plans} or \emph{couplings} between the given marginals. We denote this set by
\[
	\Pi (\mu, \nu) = \left\{\gamma \in \mathcal{P}( \R^N\times \R^N)\colon \gamma  \text{ has marginals } \mu \text{ and }\nu \right\}.
\]
We often consider processes $X=\{X_t\}_{t=1}^N$, $Y=\{Y_t\}_{t=1}^N$ defined on some probability space. % 
%\[\textstyle (\Omega, (\mathcal F_t)_{t=1}^N, \PROB)\] and taking values in $\R^N$. 
Each pair $(X,Y)$ can be thought of as a coupling or -- abusing notation slightly -- as a transport plan upon identifying it with its law. %$N$-step stochastic processes and we refer to them as couplings as well. Any property on $(X,Y)$ should then be understood as a property of the pushforward measure $(X,Y)_*\PROB$. 
For the sake of simplicity, %for us 
being measurable with respect to a sigma algebra means to be equal to a correspondingly measurable function modulo a null set w.r.t.\ the measure relevant in the given context. %By replacing $\R^N$ with $S^N$ most results apply for $S$-valued stochastic processes.
%Throughout this work most of the results would still hold for $S$-valued discrete-time  stochastic processes in $N$-steps, with $S$ a Polish space; that is, we could take $S^N$ instead of $\R^N$. 
%

%According to our convention, for a vector in $\R^N\times \R^N$ we write $x=(x_1\,\dots,x_N)$ for the first coordinate and  $y=(y_1\,\dots,y_N)$ for the second. Given a measure on $\R^N\times \R^N$, we denote by $\gamma^{x}(B):=\gamma^{x}(\{x\}\times B)$ and $\gamma^{y}(A):=\gamma^{y}(A\times\{y\})$ the regular conditional distributions w.r.t.\ respectively the first component and the second one.

%

\begin{definition}\label{def:bicausality}
A transport plan $\gamma \in \Pi(\mu,\nu)\subset \mathcal{P}( \R^N\times \R^N)$ is called \emph{bicausal} (between $\mu$ and $\nu$) if for any $B \in {\F_t}\subset \R^N$ and $t< N$, 
%$ B \in \overline{\F^{\Y}_t}^{\nu}$, 
the mappings \[\R^N \ni x\mapsto \gamma^x(B):=\gamma^x(\{x\}\times B)\text{ and } \R^N\ni y\mapsto \gamma^y(B):=\gamma^y(B\times\{y\}),\] are ${\F_t}$-measurable. 
% $\overline{\F^{\X}_t}^{\mu}$
The set of all bicausal plans is denoted \[\Pi_{bc}(\mu, \nu).\]
\end{definition}

The product measure $\mu \otimes \nu$ is bi-causal, so $\Pi_{bc}(\mu, \nu)$ is non-empty. %As in the classical setting, we shall consider the case of Borel measurable transformations $T:\R^N\to \R^N$ satisfying $T_*\mu=\nu$, so that in particular 
%\begin{equation}\label{eq:TransportMap}
%	\gamma^T:= (id\times T)_*\mu
%\end{equation} belongs to $\Pi (\mu, \nu)$ and call them (Monge) transport maps. Transport maps are termed bicausal if the associated $\gamma^T$ is so.
%
In terms of stochastic processes, a coupling is \emph{bicausal} if
\begin{align*}
	\PROB \big( (Y_1 ,\dots, Y_t)\in B_{t} \mid X_1,  \dots, X_{N} \big) &= \PROB \big( (Y_1 ,\dots, Y_t)\in B_t  \mid X_1, \dots X_t\big) \text{ and} \\ 
	\PROB \big( (X_1 , \dots, X_t)\in B_t \mid Y_1,  \dots, Y_N \big) &= \PROB \big((X_1, \dots, X_t)\in B_t \mid Y_1, \dots Y_t\big)
\end{align*}
for all $t={1,\dots, N}$ and $B_t \subset \R^t $ Borel.

%
%Heuristically this reads as ``given the past of $X$, the past of $Y$ and the future of $X$ are independent''. This is perhaps best interpreted by the following equivalent formulation (see e.g.\ \cite[Proposition 6.13]{Kallbook}):
%%
%$$\textstyle Y_t=F_t(X_1,\dots,X_t,U_t)\,\,\,\,\,\,\forall t \in \{1,\dots,N\},$$
%%
%for some measurable functions $F_t$ and where each $U_t$ is a uniform random variable independent of $X_1,\dots,X_N$.
%
%\begin{remark}
%Clearly a transport map $T$ is causal if and only if it is adapted, in the sense that there exist Borel-measurable $T^t:\R^t\to\R$ such that for $\mu$-a.e.\ $(x_1,\dots,x_N)$:
%$$\textstyle T(x_1,\dots,x_N)\,\,=\,\, (T^1(x_1),T^2(x_1,x_2),\dots,T^N(x_1,\dots,x_N)).$$
%\end{remark}
%
Testing whether a coupling or transport plan is bicausal reduces to a property of its transition kernel. Specifically we have the following characterization (see, e.g.,~\cite{BBLZ})
\begin{proposition}\label{prop:equiv charact}
	The following are equivalent:
	\begin{enumerate}
	\item  $\gamma$ is a bicausal transport plan on $\R^N \times\R^N$ between the measures $\mu$ and $\nu$.
	\item The successive regular kernels $\bar\gamma$ of the decomposition 
	\begin{align}\label{successivedisintegration}
		\gamma(\mathrm  dx_1,&\dots,\mathrm dx_N,\mathrm  dy_1,\dots,\mathrm  dy_N)\nonumber\\
		&=\bar{\gamma}(\mathrm  dx_1,\mathrm  dy_1)\gamma^{x_1,y_1}(\mathrm  dx_2,\mathrm  dy_2)\dots \gamma^{x_1,\dots,x_{N-1},y_1,\dots,y_{N-1}}(\mathrm  dx_N,\mathrm  dy_N)
	\end{align}
	satisfy
	\[\bar{\gamma}\in \Pi(p^1_*\mu,p^1_*\nu)\] and further, for $t<N$ and $\gamma$-almost all $x_1,\dots,x_{t},y_1,\dots,y_{t}$,
	\begin{equation}\label{eq: marg mu}
		p^1_*\gamma^{x_1,\dots,x_{t},y_1,\dots,y_{t}} = \mu^{x_1,\dots,x_{t}}\text{ and } p^2_*\gamma^{x_1,\dots,x_{t},y_1,\dots,y_{t}} = \nu^{y_1,\dots,y_{t}}.
	\end{equation}
	%and for $\nu$-almost all $y_1,\dots,y_{t}$
	%\begin{equation}\label{eq: marg nu}\textstyle
	%\gamma^{y_1,\dots,y_{t}}(dy_{t+1})=\nu^{y_1,\dots,y_{t}}(dy_{t+1}).
	%\end{equation}
	%\item $\gamma \in \Pi (\mu, \nu)$ and for all $ t \in\{ 1, \dots, N\}$, $h_t \in C_b(\R^t)$ and $g_t\in C_b(\R^N)$ we have 
	%%
	%\begin{align*}\textstyle
	%\int h_t(y_1,\dots,y_t)\left\{ g_t(x_1,\dots,x_{N})-  \int  g_t(x_1,\dots,x_t,\bar{x}_{t+1},\dots,\bar{x}_N)\mu^{x_1,\dots,x_t}(d\bar{x}_{t+1},\dots,d\bar{x}_{N}) \right\}d\gamma =0. 
	%\end{align*}
	%%
	%%\item $\gamma \in \Pi (\mu, \nu)$ and every $(\F^{\X}\otimes \{\emptyset,\Y\},\gamma)$-martingale is a $(\F^{\X}\otimes \F^{\Y},\gamma)$-martingale. 
	%%
	%%\comment{JB: what about continuity of M?}
	%\item $\gamma \in \Pi (\mu, \nu)$ and for every bounded continuous $\F^{\Y}$-adapted process $H$ and each bounded $\F^{\X}$-martingale $M$ we have $$\textstyle\int \sum_{t<N} H_t(y_1,\dots,y_t)\left [M_{t+1}(x_1,\dots,x_{t+1}) -  M_t(x_1,\dots,x_t)\right ] d\gamma= 0.$$
	%
	\end{enumerate}
\end{proposition}

\section{The nested distance}\label{sec:Nested}
Following \cite{Pflug2009,PflugPichler2011,PflugPichlerBuch} we consider for $\mu$, $\nu$ as above the \emph{$p$-nested distance}, or simply \emph{nested distance}, defined by
%By endowing $\P^p$ with 
\begin{align}\label{nd}
	d_{p}^{\nd}(\mu,\nu) := \left ( \inf\limits_{\gamma \in \Pi_{bc}(\mu,\nu)} \iint d^p \mathrm  d\gamma \right )^{1/p} = \left ( \inf\limits_{\gamma \in \Pi_{bc}(\mu,\nu)} \iint \left [\sum_{t=1}^N \underline{d}(x_t,y_t)^p\right ] \mathrm  d\gamma \right )^{1/p} .
\end{align}
In direct analogy with the classical $p$-Wasserstein distance it defines a metric on the space%
$$\textstyle\P^p(\R^N) := \{\mu\in\P(\R^N)\colon\, \int d(x,x_0)^p\mu(\mathrm d x)<\infty\text{ for some }x_0\}.$$
%
  % When the underlying metric $d$ is clear from context, we will simply write $d_{p}^{\nd}(\mu,\nu)$. 
 As noted in \cite{PflugPichler2016}, the nested distance \eqref{nd} is best suited to ``separate'' 
$\mu$ from $\nu$ if their information structure differs. In particular, the authors show that empirical measures $\mu^{emp}_n$ of a multivariate measure $\mu$ with density never converge in nested distance (even though they do converge in Wasserstein distance); the essential point here is that each empirical measure $\mu^{emp}_n$ is roughly a tree with non-overlapping branches (commonly a \emph{fan}) and therefore deterministic as soon as the first component is observed.
From an information perspective, $\mu_n^{emp}$ is radically different from $\mu$. Arguably, this is a key property % has to be interpreted as the essential key %strength
of the nested distance and is its main distinctive characteristic and strength in comparison with the Wasserstein distance.

\subsection{Recursive computation}
A useful comment at this point is that the nested distance can be stated and computed recursively in a way which is comparable to Bellman equations: %. Though this is a general fact, it is most transparent for $d=d_p$ given in~\eqref{eq:dp}: 
starting with $V_N^p:=0$ we define
\begin{align}\label{valuefunctiongeneral}
	&V_t^p(x_1,\dots,x_{t},y_1,\dots,y_{t}):= \\
	&\inf_{\gamma^{t+1} \in \Pi(\mu^{x_1,\dots,x_{t}}, \nu^{y_1,\dots,y_{t}} )} \iint\left(\begin{array}{cc}V^p_{t+1}(x_1,\dots,x_{t+1},y_1,\dots,y_{t+1})\\+\ \underline{d}(x_{t+1},y_{t+1})^p\end{array}\right) \gamma^{t+1}(\mathrm  dx_{t+1}, \mathrm  dy_{t+1}),\nonumber
\end{align}
%
%and by the same token, is a l.s.c.\ function of its arguments. 
% Clearly the value of \eqref{Dyn-Pbc} (equal to \eqref{Pbc} and \eqref{Dbc}) then has the recursive representation: 
so that the nested distance is finally obtained in a backwards recursive way by
\begin{equation}\label{eq:Nested}
	d_{p}^{\nd}(\mu,\nu)^p =\inf_{ \gamma^1 \in \Pi(p^1_*\mu,p^1_*\nu) }
	\iint	\left( 	V^p_1(x_1,y_1)+\ \underline{d}(x_1,y_1)^p
		\right)\gamma^1(\mathrm  dx_1,\mathrm  dy_1) .
\end{equation}

\subsection{Comparison with weak topology}
The Wasserstein distance metrizes the weak topology on probability measures with suitably integrable moments. We recall that the weak topology (also called weak* or vague topology) is characterized by integration on bounded and continuous functions. It is thus natural to ask if there is a class of functions which characterizes the topology generated by the nested distance.
% We ask: how (un)conventional is $\left(\P^p,d_{p}^{\nd}\right )$ as a metric space? Here is a first observation:

\begin{proposition}\label{prop:Separating}
Let $N\geq 2$. There does \emph{not} exist a family $\mathbb F$ of functions on $\R^N$ which determines convergence for $d_p^{\nd}$. I.e., there is \emph{no} family $\mathbb F$ so that
\[
	d_{p}^{\nd}(\mu_n,\mu)\to 0 \iff \int f\mathrm d\mu_n \to \int f\mathrm d\mu\text{ for all }f\in\mathbb F,\qquad(\mu_n)_{n=1}^\infty,\ \mu\in\P^p.
\] 
In fact, such a convergence determining family does not even exist if one restricts $d_{p}^{\nd}$ to distributions supported on a bounded region $[-K,K]^N\times [-K,K]^N$, $K> 0$.
\end{proposition}

\begin{proof} %[Proof of Proposition~\ref{prop:Separating}]
	Assume %for contradiction
	that such a family exists. Without loss of generality we can further assume that the integral of $f\in\mathbb{F}$ against all measures in $\P^p$ are well-defined. By considering $\delta_{(x_1^n,\dots,x_N^n)}$, which converge in nested distance to $\delta_{(x_1,\dots,x_N)}$ if their supports do in $\R^N$, we conclude that $\mathbb{F}\subset C(\R^N)$. Set
	\[
		\mu_\epsilon:=\frac{1}{2}\left [\delta_{(\epsilon,\dots,\epsilon,1)}+ \delta_{(-\epsilon,\dots,-\epsilon,-1)}\right ]\quad\text{ and }\quad\mu:= \frac{1}{2} \left[\delta_{(0,\dots,0,1)}+ \delta_{(0,\dots,0,-1)}\right ].
	\]
	By continuity we find that $ \int f\mathrm  d(\mu_{\epsilon}-\mu)\to 0 $ as $\epsilon \to 0$ for all $ f\in \mathbb F$. Taking $\underline{d}$ to be the usual distance on $\R$ we find $d_p^{\nd}(\mu_\epsilon,\mu)\ge 2^{p-1}$, for instance via \eqref{eq:Nested}. In general one sees that $d_p^{\nd}(\mu_\epsilon,\mu)$ is bounded away of $0$. %, since when computing the nested distance recursively we have $V_{N-1}^p(\epsilon,\dots,\epsilon,0,\dots,0)=1/2[0^p+2^p]$ and also $V_{N-1}^p(-\epsilon,\dots,-\epsilon,0,\dots,0)=1/2[0^p+2^p]$ too. 
	Thus $\mathbb{F}$ cannot determine convergence in nested distance.
\end{proof}

\begin{remark}[Separating evaluations]
	The nested distance was initially introduced with the intention to compare stochastic programs and the question addressed by the preceding Proposition~\ref{prop:Separating} was initially posed by Pflug. Indeed, Corollary~2 in~\cite{PflugPichler2011} demonstrates that there are stochastic optimization programs with differing objective values whenever the nested distance differs.
	
	The separating objects are thus entire stochastic programs which, in view of the preceding Proposition~\ref{prop:Separating}, cannot be replaced by a set of functions on $\R^N$. The proposition further emphasizes the intrinsic relation between stochastic programs, the nested distance and the role of information.
\end{remark}

We will see in Example~\ref{rem non complete} in the next section that $d^{\nd}_p$ is \emph{not}  complete. This further demonstrates how differing the nested distance and the usual Wasserstein distance are.

\begin{remark}%\todo{das würde ich weglassen.}
 We emphasize that the metric results in this and the following section %Section \ref{sec:Completeness} 
 and the topological results in Section \ref{sec:Weak} are also applicable if we based the $p$-nested distance on an $\ell^q$-type product norm in $\R$. Indeed, %if we call
%$$d_q(x,y):= \sqrt[q]{\sum_{t=1}^N \underline{d}(x_i,y_i)^q},$$
%then 
for each $q\in [1,\infty)$ we easily find $c$, $C>0$ s.t.\ $$c\,d(x,y) \leq d_q(x,y)\leq C\,d(x,y);$$ see \eqref{eq:dp} for notation. In particular, if we base the $p$-nested distance \eqref{nd} in terms of $d_q$ instead of $d=d_p$, we obtain a strongly equivalent metric on $\P^p$ (with the same constants $c$ and $C$). By the form of the metric $d$, we obtained a convenient and amenable expression for $d^{\nd}_p$, as seen in the r.h.s.\ of \eqref{nd}, which we would not have under $d_q$ for $q\neq p$. For these reasons, we may and will continue to work with $d_p^{\nd}$ defined in terms of $d=d_p$ keeping in mind how the forthcoming results are trivially generalizable. 
\end{remark}

\section{Completeness and completion}\label{sec:Completeness}
The space $\P^p(\mathbb R^N)$, endowed with the $p$-Wasserstein distance is complete. This is not the case for the nested distance, as the following example reveals. 
%
%\begin{example}\label{rem non complete}
%We observe that both $d_p^{KR}$ and $d_p^{\nd}$ are not complete metrics (ie.\ the associated metric spaces are incomplete) as soon as the dimension $N$ is greater or equal to $2$. For the sake of the argument we take $N=2$, $d=d_p$, and consider $\mu_n = 1/2\{\delta_{(1/n,1)}+\delta_{(-1/n,-1)} \}$. One verifies $d_p^{KR}(\mu_n,\mu_m)=|1/n-1/m|$, so the sequence is Cauchy for both metrics. Necessarily, the only possible %putative
%limit of this sequence can only be the limit based on the Wasserstein distance, that is $\mu =1/2\{\delta_{(0,1)}+\delta_{(0,-1)} \}$. But in nested distance we have $d_p^{KR}(\mu_n,\mu)=d_p^{\nd}(\mu_n,\mu)=(2^{p-1}+n^{-p})^{1/p}$, which does not tend to zero. 
%
%The distinguishing point is that $\mu$ is a real tree with coinciding states at the first stage, whereas the $\mu_n$'s are not. The nested distance is designed to capture this distinction, which is ignored by the Wasserstein distance.
%\end{example}

\begin{example}\label{rem non complete}
We observe that $d_p^{\nd}$ is not a complete metric as soon as the number of time steps $N$ is greater or equal than~$2$. For the sake of the argument we take $N=2$, $\underline{d}$ the usual distance on $\R$ and consider $\mu_n = 1/2\{\delta_{(1/n,1)}+\delta_{(-1/n,-1)} \}$. One verifies that $d_p^{\nd}(\mu_n,\mu_m)\leq|1/n-1/m|$, so the sequence is Cauchy. The only possible %putative
limit of this sequence is the limit based on the Wasserstein distance, that is $\mu =1/2\{\delta_{(0,1)}+\delta_{(0,-1)} \}$. But in nested distance we have $d_p^{\nd}(\mu_n,\mu)=(2^{p-1}+n^{-p})^{1/p}>1$, in particular this sequence does not tend to zero. 

The distinguishing point is that $\mu$ is a real tree with coinciding states at the first stage, whereas the $\mu_n$'s are not. The nested distance is designed to capture this distinction, which is ignored by the Wasserstein distance.
\end{example}

So for $N>1$ the nested distance is not complete. To identify the completion of $\P^p(\mathbb R^N)$ with respect to the $p$-nested distance we consider the \emph{nested distributions} introduced in~\cite{Pflug2009}.

\begin{definition}\label{def:Nested}
	Consider the sequence of metric spaces
	\begin{align*}
		R_{N:N}&:= (\mathbb R, d_{(N:N)}), \text{ with } d_{(N:N)}=\underline{d}=[\underline{d}^p]^{1/p},\\
		R_{N-1:N}&:= \big(\mathbb R\times \mathcal P^p(R_{N:N}), d_{(N-1:N)}\big), \text{ with }d_{(N-1:N)}=\left [\underline{d}^p + W_{d_{(N:N)},p}^p\right]^{1/p},\\
		&\vdots\nonumber\\
		R_{1:N}&:= \big((\mathbb R\times \mathcal P^p(R_{2:N})), d_{(1:N)}\big), \text{ with }d_{(1:N)}=\left[\underline{d}^p + W_{d_{(2:N)},p}^p\right]^{1/p},
	\end{align*}
	where at each stage $t$, the space $\mathcal P^p(R_{t:N})$ is endowed with the $p$-Wasserstein distance with respect to the metric $d_{(t:N)}$ on $R_{t:N}$, which we denote $W_{d_{(t:N)},p}$. %is the set of all probability measures with finite moment $\int d(x_0,x)P(\mathrm dx)<\infty$ and $d_t$ the distance on the product space. 
	The set of \emph{nested distributions} (of depth $N$) with $p$-th moment is defined as $\mathcal P^p(R_{1:N})$.
\end{definition}

%\comment{Aqui quede}
Each of the spaces $R_{t:N}$ ($t=1,\dots N$) is a Polish space. Indeed, a complete metric is given explicitly and the spaces are separable since $\mathcal P(R)$ is complete and separable whenever $(R,\rho)$ is complete and separable (cf.\ \cite{Bolley2008}). We endow $\mathcal P^p(R_{1:N})$ with the complete metric $W_{d_{(1:N)},p}$.

\begin{example}\label{ex computable}
When $N=2$, we have that $R_{1:2}=\mathbb{R}\times \mathcal{P}^p(\mathbb{R})$ and for $P,Q\in \mathcal P^p(R_{1:2})$ the distance is 
\begin{align}\label{eq simple nested} \textstyle
	W_{d_{(1:2)},p}(P,Q)= \left\{\inf_{\Gamma\in \Pi(P,Q)} \iint\big(\underline{d}(x,y)^p + W_p^p(\mu,\nu)  \big)\Gamma( \mathrm dx,\mathrm d\mu,\mathrm dy, \mathrm d\nu) \right\}^{1/p}
\end{align}  
with $W_p$ the classical $p$-Wasserstein distance for measures on the line and w.r.t\ the metric $\underline{d}$. %Thus $ W_{d_{(1:2)},p}$ is a Wasserstein distance of its own.
The formulation~\eqref{eq simple nested} notably exactly corresponds to the recursive descriptions~\eqref{valuefunctiongeneral} and~\eqref{eq:Nested}.
\end{example}

\subsection{Embedding}
We demonstrate that the nested distributions of depth $N$ introduced in Definition~\ref{def:Nested} extend the notion of probability measures in $\mathbb R^N$ in a metrically meaningful way. Let us introduce the following function, already present in \cite{Pflug2009}, which associates $\mu\in\mathcal P^p(\mathbb R^N)$ with the measure $I[\mu]\in \mathcal P^p(R_{1:N})$ given by
\begin{equation}\label{eq I}
	I[\mu]:=\L\left(X_1\,,\, \L^{X_1}\Big(X_2\,,\, \cdots % \,,\, \L^{X_{1:N-3}}\left( X_{N-2}
	\,,\, \L^{X_{1:N-2}}\big(X_{N-1}\,,\, \L^{X_{1:N-1}}(X_N)\big )\Big)\right)%\right)
	,
\end{equation}
where $(X_1,\dots,X_N)$ is a vector with law $\mu$. We used the shorthand $\L^{X_{1:k}}$ for the \emph{conditional law} given $(X_1,\dots,X_k)$ (and no superscript indicates unconditional law).

%\comment{JB: We should cite result about classical Wasserstein space being separable}
\begin{theorem}\label{thm: embedding} Let $d=d_p$. Then the classical Wasserstein distance of nested distributions extends the nested distance of classical distributions. %Then the nested distance of classical distributions is extended by the classical Wasserstein distance of nested distributions.
More precisely, the mapping~$I$ defined in~\eqref{eq I} embeds the metric space $(\mathcal P^p(\mathbb R^N),d_p^{\nd})$ defined via \eqref{nd} \emph{isometrically} into the separable complete metric space $(\mathcal P^p(R_{1:N}), W_{d_{(1:N)},p})$.
%More precisely: the metric space $(\mathcal P^p(\mathbb R^N),d_p^{\nd})$ defined via \eqref{\ndp} is isometrically embedded, by $I$ from \eqref{eq I}, into the separable complete metric space $(\mathcal P^p(R_{1:N}), W_{d_{(1:N)},p})$.
In particular $(\mathcal P^p(\mathbb R^N),d_p^{\nd})$ is separable. 
\end{theorem}
\begin{proof}
	It is enough to consider $N=2$. For a probability measure $\mu$ on $\mathbb R^2$ consider its disintegration measure
	\begin{equation}\label{eq:7}\notag \textstyle
		\mu(A\times B)= \int_A \mu^{x_1}(B) p^1_*\mu(\mathrm d x_1),
	\end{equation}
	where $p^1$ is the projection onto the first coordinate.
	An embedding of $\mu$ in the space $\mathcal P\big(\mathbb R\times \mathcal P(\mathbb R)\big)$ is given by the probability measure generated uniquely by (here $A,B$ are Borel sets of $\mathbb R$ and $\mathcal P(\mathbb R)$ respectively)
	\[I[\mu](A\times B):=\mu \big(A\cap T_{\mu}^{-1}(B)\big)= p^1_*\mu \big(A\cap T_{\mu}^{-1}(B)\big),\]
	where $T_{\mu}$ is the Borel measurable function
	\begin{align*}
		T_{\mu}: \mathbb R&\to \mathcal P(\mathbb R)\\
		x_1&\mapsto \mu^{x_1}(\mathrm dx_2)
	\end{align*}
%	(Cf.\ \eqref{eq:TransportMap}).
	In this way we find that $I[\mu]$ is the $\mu$-law of $x_1\mapsto (x_1,\mu^{x_1}(\mathrm d x_2))$. For $\mu\in \mathcal P^p(\mathbb R^2)$ we also have
	\begin{align*}
	\int\left \{ \underline{d}(x,0)^p+W_p^p(\nu,\delta_0) \right \} I[\mu](\mathrm dx,\mathrm d\nu)&= \int\left \{ \underline{d}(x_1,0)^p+W_p^p(\mu^{x_1},\delta_0) \right \} p^1_*\mu(\mathrm dx_1)\\
	&= \int\left \{ \underline{d}(x_1,0)^p+\underline{d}(x_2,0)^p \right \}\mu(\mathrm dx_1,\mathrm dx_2)<\infty
	\end{align*}
	%induced by the relation~\eqref{eq:7}
	and thus $I[\mu]\in  \mathcal P^p(R_{1:2})$.

	We now observe that the embedding $\mu\mapsto I[\mu]$ is actually an isometry between $(\mathcal P^p(\mathbb R^N),d_p^{\nd})$ and $(\mathcal P^p(R_{1:N}), W_{d_{(1:N)},p})$. To this end, first note that every coupling between $I[\mu]$ and $I[\nu]$ (i.e., every $\Gamma \in \Pi(I[\mu],I[\nu])$) is of the form $\bar{\gamma}(\mathrm dx_1,\mathrm dy_1)\delta_{T_{\mu}(x_1)}(\mathrm d M)\delta_{T_{\nu}(y_1)}(\mathrm d N)$ for some $\bar{\gamma}\in \Pi(p^1_*\mu,p^1_*\nu)$ and vice-versa. Hence from~\eqref{eq simple nested} and~\eqref{valuefunctiongeneral} we have that
	\begin{align}
		W_{d_{(1:2)},p}\big(I[\mu],I[\nu]\big)^p&= \inf_{\bar{\gamma}\in \Pi(p^1_*\mu,p^1_*\nu)} \int\left \{ \underline{d}(x_1,y_1)^p + W_p^p(\mu^{x_1},\nu^{y_1}) \right   \}\bar{\gamma}(\mathrm dx_1,\mathrm dy_1)\notag \\
		&= d_p^{\nd}(\mu,\nu)^p, \label{eq intermediate lift}
	\end{align}
by \eqref{eq:Nested} and hence the isometry. Finally, since the image of $I$ is a subspace of the separable metric space $(\mathcal P^p(R_{1:N}), W_{d_{(1:N)},p})$, it is separable itself. We conclude that $(\mathcal P^p(\mathbb R^N),d_p^{\nd})$ is separable too.
\end{proof}

\begin{remark}
From the preceding arguments follows that the embedding~$I$ in~\eqref{eq I} is onto if and only if $N=1$. 
\end{remark}

\begin{remark}
Returning to Example \ref{ex computable} and applying \eqref{eq intermediate lift} we find for $\mu,\nu\in\mathcal P^1(\R^2)$
\begin{align*}
 d_1^{\nd}(\mu,\nu)&= W_{d_{(1:2)},1}\big(I[\mu],I[\nu]\big)\\
& =\textstyle \sup\left\{ \int F(x,\mu^x)p^1_*\mu(\mathrm d x) -\int F(x,\nu^x)p^1_*\nu(\mathrm d x)  \right \},
\end{align*}
where the supremum is taken over all (bounded) functions $F:\R\times\mathcal P^1(\R)\to\R$ with Lipschitz constant at most one (with respect to the metric $\underline{d}+ W_1$). Indeed, this is nothing but the Kantorovich-Rubinstein Theorem (\cite[Theorem 1.14]{Vi03}) for the 1-Wasserstein metric on $\mathcal P^1(\R\times\mathcal P^1(\R))$. Similar results apply for $\mu,\nu\in\mathcal P^1(\R^N)$ by using $R_{1:N}$ instead.

\end{remark}

\subsection{Completion}

We now identify %establish that  $(\mathcal P^p(R_{1:N}), W_{d_{(1:N)},p})$ is the 
the completion of $(\mathcal P^p(\mathbb R^N),d_p^{\nd})$. (Recall that the completion of a metric space is unique up to isomorphism.)
This result provides a solid link between these two previously separate mathematical objects.

\begin{theorem}\label{Thm completion}
The space $\big(\mathcal P^p(R_{1:N}), W_{d_{(1:N)},p}\big)$ od nested distribution is the completion of $\big(\mathcal P^p(\mathbb R^N),d_p^{\nd}\big)$.
\end{theorem}

\begin{proof}
We  need  to provide an isometry $J$ from $(\mathcal P^p(\mathbb R^N),d_p^{\nd})$ into  $ (\mathcal P^p(R_{1:N}), W_{d_{(1:N)},p})$
whose range is dense. %In the proof Theorem~\ref{thm: embedding} we have built an isometry $I$ for $N=2$ (in case $N=1$ one takes the identity).
We shall prove that $I$ defined in~\eqref{eq I} does this task. This can be done for  arbitrary $N$ at a notational, while already the case $N=2$ is representative of the general situation. We thus assume $N=2$ in what follows. 

The set of convex combinations of Dirac measures is dense in $\mathcal P^p(R_{1:N})$ w.r.t.\ the metric $ W_{d_{(1:N)},p}$. This is actually true for any Wasserstein metric (cf.\ \cite{Bolley2008}) and thus particularly for $W_{d_{(1:2)},p}$, which in itself is a Wasserstein metric (see also Example~\ref{ex computable} for concreteness). So it suffices to prove that convex combinations of Dirac measures lie in the closure of the range of $I$.

Let $A:=(a_1,\dots,a_k)$ be a $k$-tuple of points in $\mathbb R$ and $m_1,\dots,m_k$ be measures on the line with finite $p$-th moment. Given weighs $\{\lambda_i\}_{i=1}^k$ we are interested in the measure $$P(\mathrm dx, \mathrm dm)=\sum \lambda_i \delta_{(a_i,m_i)}(\mathrm dx,\mathrm dm),$$ over $R_{1:2}$. Now we take any sequence  $A^n:=\{a^n_1,\dots,a^n_k\}$ such that componentwise $A^n\to A$ as $n\to \infty$ and, for each $n$ fixed, all coordinates of $A^n$ are distinct.
%
% Clearly we may assume that the first $k_1$ elements in $A$ equal the same value $b_1$, the next $k_2$ elements of $A$ equal $b_2$ and so until the remaining $k_s$ elements equal $b_s$, for some $s\leq k$ and with $k_1+\dots +k_s=k$, with the $b$'s all different. For given $n\in\mathbb N$ we define \begin{equation}
%\hat{a}_j= b_i + \frac{k_1+\dots +k_i-j+1}{n}, \label{eq aux}
%\end{equation}  if $j\in ( \sum_{v\leq i-1} k_v , \sum_{v\leq i} k_v ]$ for some (unique) $i\in\{1,\dots,s\}$ and otherwise take $i=1$ in \eqref{eq aux}. This is to say, each $a_j$ which is equal to a (unique) $b_i$ is added a quantity $c/n$ according to its order of appearance in $A$ (the first such $a_j$ gets a larger $c$, the next ones an ever diminishing $c$). In this way, if $n$ is large enough we have that all the $\hat{a}$'s are distinct and of course $\hat{a}_j\to a_j$ as $n\to\infty$.
%
We now define $\mu_n\in \mathcal P^p(\mathbb R^2)$ as the measure whose first marginal is $\sum \lambda_j \delta_{{a}^n_j}$ and such that $\mu_n(\mathrm dx_2|x_1={a}^n_j)=m_j(\mathrm dx_2)$.
It is elementary, and this is the main point of having made the ${a}^n_j$'s distinct for a fixed $n$, that \[I[\mu_n]= \sum \lambda_j\delta_{({a}^n_j,m_j)}.\]
Consequently we get that $I[\mu_n]\to P$ with respect to $ W_{d_{(1:2)},p}$ when $n\to\infty$, as desired.
%
%The proof is complete if we show the density of the set of convex combinations of Dirac measures. But this is true for any Wasserstein metric (cf.\ \cite{Bolley2008}), and as we have observed in Example \ref{ex computable},  $ W_{d_{(1:2)},p}$ is in itself a Wasserstein metric.
\end{proof}

%\subsection{Remaining problems}
%At this stage we want to raise the following two questions, which are still open after the careful analysis presented in this section.
%
%\begin{enumerate}
%	\item  Can $(\mathcal P^p(\mathbb R^N),d_p^{\nd})$ still be Polish, for $N\geq 2$?
%	\item Is there an analogue to Prokhorov's Theorem on $\big(\mathcal P^p(\mathbb R^N),d_p^{\nd}\big)$?
%\end{enumerate}

%{\color{blue}

At this point we ask:

 {\begin{center}can $(\mathcal P^p(\mathbb R^N),d_p^{\nd})$ still be Polish, for $N\geq 2$?\end{center}	}
\noindent Just as the interval $(0,1)$ is a Polish subspace of $[0,1]$, which is nevertheless incomplete w.r.t.\ the usual distance on $[0,1]$, neither Example~\ref{rem non complete} nor Theorem~\ref{Thm completion} contribute anything to this question. We explore this in the following section.
%\end{remark}
%}

%In the next section we introduce the new concept of weak nested convergence. This will lead us to answer the first point above to the negative. As a consequence, the second point above can only be realized after embedding usual distributions into nested distributions.

\section{The information topology/\,weak nested topology}\label{sec:Weak}

We introduce the space $\P(R_{1:N})$ just as we did for $\P^p(R_{1:N})$, but now denoting $R_{t-1:N}:=\mathbb R\times \mathcal P(R_{t:N})$ at each step of the recursive definition and equipping $R_{t-1:N}$ with the product topology of Euclidean distance in the first component and the usual weak topology in the second one. Doing so, we conclude that $\P(R_{1:N})$ is a Polish space of measures on the likewise Polish space $R_{1:N}$. Inspired by the isometric embedding in Theorem~\ref{thm: embedding}, which we denoted $I$ in \eqref{eq I}, a mapping $I:\P(\R^N)\to \P(R_{1:N})$ can be obtained by direct generalization. %We are interested in:

\begin{definition}\label{def:Information}
We say that a net $\{\mu_\alpha\}_\alpha$ in $\P(\R^N)$ converges \emph{weakly nested} to $\mu\in\P(\R^N)$, if and only if $I[\mu_\alpha]$ converges weakly in $\P(R_{1:N})$ to $I[\mu]$. We %write $\mu_\alpha \overset{w}\Rightarrow    \mu$ and 
call the corresponding topology \emph{weak nested topology}. 
\end{definition}

Thus the weak nested topology is simply the initial topology for the map $I$.

\begin{remark}\label{rem Hell}
Suppose that $X$ is Polish and that $\rho$ is a compatible complete metric which is \emph{bounded}. Then the corresponding $p$-Wasserstein topology is precisely the weak topology.  Likewise, we obtain that the weak nested topology is generated by the nested distance $d^\nd_{p}$ as soon as we choose $\underline{d}$ as a compatible bounded metric for the usual topology on $\R$. For instance, we may take  
 the metric $\underline{d}(a,b)=|a-b|\wedge 1$ so
 \begin{align}\notag
  d(x,y)= \sum_{t=1}^N |x_t-y_t| \wedge 1. 
 \end{align}
In this way we obtain that the weak nested topology coincides with a $p$-nested topology of the form we have already treated.
 
\end{remark}

Although there are more direct ways to prove it, the previous remark implies the following:

\begin{lemma}
The weak nested topology is separable and metrizable. % Specifically, there is a countable family $\{F_n\}_n$ of bounded continuous functions on $R_{1:N}$ determining weak nested convergence. 
\end{lemma}

{

\subsection{Comparison with an existing concept}\label{comparison Hellwig}

%\begin{remark}
	Definition~\ref{def:Information} is related to the so-called \emph{topology of information} which was  introduced in \cite{Hellwig} for the purpose of sequential decision problems and equilibria (see also \cite{Barbie} for a recent update). In our setting, this topology is defined as the initial topology for the following maps on $\mathcal{P}(\R^N)$:
	\begin{align*}
	\mu&\mapsto \mu \,\,\,\in \P(\R^N),\\ 
	\mu& \mapsto [x_1\mapsto(x_1,\mu^{x_1}(\mathrm dx_2,\dots, \mathrm dx_N))]_*\mu \,\,\,\in \P(\R\times \P(\R^{N-1})),\\
		\mu& \mapsto [(x_1,x_2)\mapsto(x_1,x_2,\mu^{x_1,x_2}(\mathrm dx_3,\dots, \mathrm dx_N))]_*\mu \,\,\,\in \P(\R^2\times \P(\R^{N-2})),\\
	&\vdots\\
	\mu& \mapsto [(x_1,\dots,x_{N-1})\mapsto(x_1,\dots,x_{N-1},\mu^{x_1,\dots,x_{N-1}}(\mathrm dx_N))]_*\mu \,\,\,\in \P(\R^{N-1}\times \P(\R)) ,
	\end{align*}
	where the range spaces are endowed with the usual weak topologies.

	For $N=1,2$ this topology obviously coincides with the weak nested one of Definition \ref{def:Information}. As a matter of fact, this is always the case. We demonstrate the argument for $N=3$:
	
Let $A,B$ continuous bounded function on $\R\times \P(\R^2)$ and $\R^2\times \P(\R)$, respectively. We denote by $m,M$ generic elements in $\P(\R)$ and $\P(\R\times \P(R))$, respectively. Then
\begin{align*}
	\R\times \P(\R\times \P(R))\ni (x_1,M)& \mapsto \bar{A}(x_1,M):=\, A\left(x_1, \int_{\R\times\P(\R)} m\, M(\mathrm d x_2,\mathrm d m) \right),\\
	 \R\times \P(\R\times \P(R))\ni (x_1,M)& \mapsto \bar{B}(x_1,M):=\,
	\int_{\R\times\P(\R)} B(x_1,x_2,m)M(\mathrm d x_2, \mathrm dm), 
\end{align*}	 	
	 are seen to be continuous bounded functions too, thus they are suitable test functions for weak nested convergence (since $R_{1:3}=\R\times \P(\R\times \P(R))$ precisely). One then easily verifies that
	 \begin{align*}
	\int \bar{A}(x_1,M)I[\mu](\mathrm dx_1,\mathrm d M)&= \int_{\R} A(x_1,\mu^{x_1}(\mathrm dx_2,\mathrm dx_3))\mu(\mathrm d x_1) , \\
	\int \bar{B}(x_1,M)I[\mu](\mathrm dx_1,\mathrm d M)&= \int_{\R^2} B(x_1,x_2,\mu^{x_1,x_2}(\mathrm dx_3)) \mu(\mathrm d x_1,\mathrm d x_2) , 
\end{align*}	 
so convergence of the l.h.s\ (guaranteed by weak nested convergence) implies that of the r.h.s.\ which then implies convergence in information topology. Thus the weak nested topology is stronger than the topology of information. For the converse, recall first that 
$$\mu\in \P(\R^3)\mapsto I[\mu]=\L(X_1\, ,\, \L^{X_1}(X_2 \, ,\, \L^{X_1,X_2}(X_3))),$$
where $(X_1,X_2,X_3)$ is distributed according to $\mu$. If we denote
\begin{align*}
m\in \P(\R^2)&\mapsto T[m]:= [x\mapsto (x,m^x(\mathrm dy))]_*m\in \P(\R\times\P(\R)),\\
(x,m)\in \R\times \P(\R^2) &\mapsto L(x,m):= (x,T[m])\in \R\times\P(\R\times\P(\R)),\\
p\in \P(\R^3)&\mapsto \phi(p):=[x\mapsto (x,p^x(\mathrm dy,\mathrm dz))]_*p\in \P(\R\times\P(\R^2)),
\end{align*}
one finds that $I[\mu]=L_*\phi(\mu)$. By definition $\phi $ is continuous in topology of information. The key now is~\cite[Lemma 7]{Hellwig}, from which $\phi$ is also continuous if on the range space,  $\P(\R\times\P(\R^2))$, we endow $\P(\R^2)$ with the information topology again. Since $L$ is continuous when the domain is given this topology, we finally conclude that $L_*\phi(\mu)$ is continuous in information topology, making the latter stronger than the weak nested one.

Because we are inspired by the nested distance of Pflug and Pichler and due to the observation in Remark \ref{rem Hell} that the weak nested topology (so a fotriori the information topology) is a nested distance topology, we shall use the term ``weak nested topology'' instead of ``information topology'' in the following. The results to come are also new in the setting of \cite{Hellwig,Barbie}.

%The results to come on the weak nested topology are new 
	 
%At the moment we conjecture, but do not know for sure, that actually both topologies coincide. Whatever the case is, 

%the main result to come (Theorem \ref{thm polish}) is fully novel.
%\end{remark}

}

\subsection{A closer look at the weak nested topology}

We will establish that the weak nested topology (and actually the nested distance topologies) is Polish. %Although we deem this interesting per se, we hope this can find future applications too.

%\begin{lemma}
%The weak nested topology is separable and metrizable. % Specifically, there is a countable family $\{F_n\}_n$ of bounded continuous functions on $R_{1:N}$ determining weak nested convergence. 
%\end{lemma}
%\begin{proof}
%By definition weak nested convergence induces the weakest topology for which  $I$ is continuous. Furthermore, $I$ is an embedding (i.e., a homeomorphism between $\P(\R^N)$ and $I[\P(\R^N)]$ with the relative topology). One concludes by the classical arguments in \cite{Varadarajan} and by recalling that subspaces of metric separable spaces are separable again.
%\end{proof}

We recall that a set of a topological spaces is a $G_\delta$ if it is the countable intersection of open sets. Recall also that every separable metrizable space is homeomorphic to a subspace of the Hilbert cube $[0,1]^\NN$, the latter equipped with the product topology; see \cite[Theorem 4.14]{kechris}. A compatible metric on the Hilbert cube is given by \[D\big((x_n),(y_n)\big):=\sum_{n=1} 2^{-n}|x_n-y_n|.\]

\begin{lemma}
\label{lem graph}
Let $m\in\P(X\times Y)$ with $X$ Polish and $(Y,\rho)$ a separable metric space. Denote $\iota :Y\to [0,1]^\NN$ the embedding of $Y$ into the Hilbert cube. Then the following are equivalent:
\begin{enumerate}
	\item \label{enu:11} $m\big(\mbox{Graph}(f)\big)=1$ for $f:X\to Y $ Borel;
	\item \label{enu:22}$\inf\left\{\int_{X\times Y} \rho\big(f(x),y\big)  m(\mathrm dx,\mathrm dy) :\,\, f:X\to Y \text{ Borel} \right\} = 0$;
	\item \label{enu:33}$\inf\left\{\int_{X\times Y} D\big(F(x),\iota(y)\big)  m(\mathrm dx,\mathrm dy) :\,\, F:X\to [0,1]^\NN \text{ Borel} \right\} = 0$;
	\item \label{enu:44}$\inf\left\{\int_{X\times Y} D\big(F(x),\iota(y)\big)  m(\mathrm dx,\mathrm dy) :\,\, F:X\to [0,1]^\NN \text{ continuous} \right\} = 0$.
\end{enumerate}
\end{lemma}

\begin{proof}
Clearly \ref{enu:11}$\implies$\ref{enu:22}$\implies$\ref{enu:33}. Denote $\mu$ the first marginal of $m$. Given  $F:X\to [0,1]^\NN $ Borel, $F(x)=(F_n(x))_n$, we can approximate it in $L^1(X,\mu;[0,1]^\NN)$ by continuous functions. This follows since coordinate-wise we can approximate $F_n\in L^1(X,\mu;[0,1])$ by continuous functions. So also \ref{enu:33}$\implies$\ref{enu:44}.

To establish \ref{enu:44}$\implies$\ref{enu:11} let $\{F_n\}$ be a sequence of continuous functions approximating the infimum in~\ref{enu:44} and denote $G_n(x):=\int D\big(F_n(x),\iota(y)\big)m^x(\mathrm dy)$ so $G_n$ is Borel, non-negative and $\|G_n\|_{L^1(X,\mu;\R)}\to 0$ by definition. It follows that $G_n\to 0$ in $L^1(X,\mu;\R)$ so up to a subsequence $G_n(x)\to 0$ for $\mu$-a.e. $x$. From now on we work on such a full measure set, on which we can further assume that $m^x\in\P(Y)$. Assume that %on a non-negligible further subset $\Omega$, 
we had that $|supp(m^x)|>1$. Then there would exist disjoint compact sets $K^1_x$, $K^2_x\subset Y$ with $M_x=\min\{m^x(K^1_x),m^x(K^2_x)\}>0$. Obviously $\iota(K^1_x),\iota(K^2_x)$ are also disjoint compact sets, so $D_x:=D\big(\iota(K^1_x),\iota(K^2_x)\big)>0$. By the triangle inequality, $\max\{D\big(F_n(x),\iota(K^1_x)\big),D\big(F_n(x),\iota(K^2_x)\big)\}\geq D_x/2 $, thus $G_n(x)\geq M_xD_x/2$, % for $x\in \Omega$, 
yielding a contradiction. We conclude that $\mu$-a.s.\ $|supp(m^x)|= 1$ and therefore %whenever $m^x$ is non-trivial
 we must have $\iota\big(f(x)\big):=\lim_n f_n(x)$ exists, for some $f:X\to Y$ Borel. Thus $m^x(\mathrm dy)=\delta_{f(x)}(\mathrm dy),\mu-a.s.$, which proves~\ref{enu:11}.
\end{proof}

Observe that it is crucial for point~\ref{enu:44} in Lemma \ref{lem graph} that we embedded $Y$ in the Hilbert cube. Indeed, it $X$ is connected and $Y$ discrete, the only continuous functions $f:X\to Y$ are the constants. We now present a result which is interesting in its own:

\begin{proposition}\label{lem polish crucial}
Let $X$ and $Y$ be Polish spaces. Then \[S\,:=\,\{m\in\P(X\times Y)\colon m\big(\text{Graph}(f)\big)=1,\,\, \mbox{some Borel }f\colon X\to Y\,\},\]
with the relative topology inherited from $\P(X\times Y)$, is Polish too.
\end{proposition}

\begin{proof}
Let $\rho$ be a compatible metric for $Y$, which we may assume bounded.
By Lemma \ref{lem graph} we have 
\begin{align}\label{eq S}
\textstyle S= \bigcap_{n\in\mathbb N}\bigcup\limits_{F:X\to [0,1]^\NN \mbox{ continuous}}\{m\in \P(X\times Y)\colon \int D\big(F(x),\iota(y)\big)  m(\mathrm dx,\mathrm dy)< 1/n\},
\end{align}
where $\iota:Y\to [0,1]^\NN$ is an embedding. Since $(x,y)\mapsto D\big(F(x),\iota(y)\big)$ is continuous bounded if $F$ is continuous, the set in curly brackets is open in the weak topology. Thus the union of these is open too and we get that $S$ is a $G_\delta$ subset. We conclude by employing \cite[Theorem 3.11]{kechris}.
\end{proof}

\begin{theorem}\label{thm polish}
The weak nested topology on $\P(\R^N)$ is Polish.
\end{theorem}
\begin{proof}
%The same argument in Proposition \ref{prop crucial} can be applied here recursively. Indeed, 
%We argue recursively. First, $\P(R_{N:N})=\P(\R)$ with the weak convergence is Polish. Then, assuming that we have proved that $\P(R_{k:N})$ is Polish, one goes on to establish that $\P(R_{k-1:N})=\P(\R\times \P(R_{k:N}))$ is Polish too. For this, it suffices to observe that $\P(R_{k:N})$ is homeomorphic to $\R$.

For $N= 2$ we have $\P(R_{1:2})=\P(\R\times\P(\R))$ and, by definition, $\P(\R^2)$ equipped with the weak nested topology is homeomorphic to $I[\P(\R^2)]$ equipped with the relative topology inherited from $\P(R_{1:2})$. We have
\[I[\P(\R^2)] = \left\{P \in \P(\R\times\P(\R)):\,\, P\big(\text{Graph}(f)\big)=1,\text{ some Borel }f:\R\to \P(\R)\,  \right\}. \]
To wit, if $P\in I[\P(\R^2)]$, then by definition of the embedding $I$ we have $P=(id,T)_*(p^1_* \mu)$ for some $\mu\in\P(\R^2)$ and $T(x)=\mu^x$ (see also the proof of Theorem \ref{thm: embedding}). Taking $f=T$ we then get that $P$ belongs to the right hand side above. Conversely, given $P$ in the right hand side, we denote by $\mu_1$ its first marginal and define $\mu^x(dy):=\delta_{f(x)}(dy)$. The measure $\mu(dx,dy):=\mu_1(dx)\mu^x(dy)\in\P(\R^2)$ satisfies $I[\mu]=P$.

By Proposition \ref{lem polish crucial} we conclude that $I[\P(\R^2)]$ is Polish and then so is $\P(\R^2)$, as desired. The case for general $N$ is identical; one observes by reverse induction that if $\P(R_{t:N})$ is Polish, then so is $\P(R_{t-1:N})$ using the above arguments.
%
%The case $N=2$ follows from  Proposition \ref{prop crucial} with $Y=\P(\R)$ and by the fact that $\P(R_{1:2})=\P(\R\times\P(\R))$. The other cases by tedious induction and repeated use of Proposition \ref{prop crucial}.
\end{proof}

{
Proposition \ref{lem polish crucial}, and more specifically \eqref{eq S}, permit to actually find a compatible complete metric for the weak nested topology. The embedding into the Hilbert cube would make such metric look more complicated than necessary. As we argue now, there is a way to identify a slightly less abstract compatible complete metric. 
%One can also prove Theorem \ref{thm polish} by a suitable modification of Lemma \ref{lem graph}, with the advantage of obtaining a complete compatible metric for the Polish topology, and this we do next. With similar arguments, a more involved complete metric can also be found via Proposition \ref{lem polish crucial}, as the proof will reveal.  %In the next result we identify one complete metric which is compatible with (i.e.\ it generates) the weak nested topology.
 For simplicity of notation we just consider $N=2$ here:

\begin{corollary}\label{coro complete metric}
Let $\rho$ be a bounded metric compatible with the weak topology on $\P(\R)$ and $d^w$ a complete metric compatible with the weak topology on $\P(\R\times \P(\R))$. Then the weak nested topology on $\P(\R^2)$ is generated by the complete metric
\begin{align}\label{eq complete metric}
d^{wnt}(P,Q):= d^w(I[P],I[Q])+ \sum_{n\in\mathbb{N}} 2^{-n}\wedge \left | \frac{1}{d^w(I[P],A_n)} - \frac{1}{d^w(I[Q],A_n)} \right | ,
\end{align}
where $A_n:=\{m\in \P(\R\times \P(\R)):\, \int \rho(F(x),y)m(\mathrm dx,\mathrm dy)\geq 1/n ,\,\forall \,\, F:\R\to\P(\R)\,\,\mbox{continuous}\}$, with the embedding $I$ as in \eqref{eq I} and
$$ d^w(\cdot,A_n):= \inf_{m\in A_n} d^w(\cdot,m) .$$
\end{corollary}

\begin{proof}
%We first observe that for Lemma \ref{lem graph} and $Y=\P(\R)$, we can by-pass the embedding into the Hilbert cube by noticing that this space is actually isomorphic to $\P(\R)$. 
We first observe that for Lemma \ref{lem graph} and $Y=\P(\R)$, we can by-pass the embedding into the Hilbert cube. One way to do this, is to follow the ``Tietze extension'' argument in the proof of \cite[Proposition C.1]{CDLmeanfield}, establishing the equivalence of (1) and (4) in Lemma \ref{lem graph} where now the continuous functions go from $X=\R$ to $\P(\R)$. We can thus write \eqref{eq S}, in the case $Y=\P(\R)$, without the embedding $\iota$. Using this and following the proof of \cite[Theorem 311]{kechris} we find a compatible complete metric for $I[\P(\R^2)]$ with the relative topology inherited from $\P(\R\times\P(\R))$, via:
$$\textstyle I[\P(\R^2)]\ni\bar{P},\bar{Q}\mapsto d^w(\bar{P},\bar{Q})+ \sum_{n} 2^{-n}\wedge \left | d^w(\bar{P},A_n)^{-1} - d^w(\bar{Q},A_n)^{-1} \right |.$$
This is then transformed into a complete metric for  $\P(\R^2)$ via the homeomorphism $I$, yielding \eqref{eq complete metric}.
\end{proof}

\begin{remark}
Notice that Example~\ref{rem non complete} shows that the weak nested topology is strictly stronger than the weak topology for $N\geq 2$. In such case, it also shows that even if a sequence of measures has their support contained in a common compact, there need not exist a convergent subsequence, unlike in the weak topology. %It could be interesting and non-trivial, to characterize the relatively compact sets of the weak nested topology.
\end{remark}

 Analogous considerations show that $\P^p(R_{1:N})$ with the $p$-nested distance is Polish as well. Having established the completion and the Polish character of the $p$-nested distance, it remains an open question, whether there is a more amenable compatible complete metric than the one found in Corollary \ref{coro complete metric}. %Instead, we now investigate a different and appealing notion of distance, which we compare to the nested distance. }

\section{Extreme points of related sets}
\label{sec:extreme}

Recall the definition of \emph{bicausal} transport plans $\Pi_{bc}(\mu, \nu)$. With notation as in Definition \ref{def:bicausality}, we also say that $\gamma\in\Pi(\mu,\nu)$ is causal if  
%$ B \in \overline{\F^{\Y}_t}^{\nu}$, 
the mappings
\[\R^N \ni x\mapsto \gamma^x(B):=\gamma^x(\{x\}\times B),\] are ${\F_t}$-measurable for any $B \in {\F_t}\subset\R^N$ and $t< N$. This is a weaker condition than bicausality and the set of such couplings is denoted  \[\Pi_{c}(\mu, \nu).\] We will write $\Pi(\mu, \cdot)$ meaning that the second marginal of these couplings is left unspecified, with similar notation for the causal and bicausal case.

We are interested in determining the extreme points of the convex sets
$$ \Pi_{c}(\mu, \cdot) \,\,\,\mbox{ and } \,\,\, \Pi_{bc}(\mu, \cdot). $$
We expect this knowledge to find applications in (discrete-time) stochastic optimization. This will be explored elsewhere. It suffice to say, as in the introduction, that such extreme points are expected to play an important role when one is interested in ``simplifying'' the process law $\mu$ without changing its information structure. Let
 \begin{align}\label{def Monge c}
\Pi_{c}^{Monge}(\mu, \cdot)&:= \left\{ \gamma=(id, T)_*\mu\colon\, T\colon\R^N\to\R^N \mbox{ is Borel and adapted}  \right \},\\
\Pi_{bc}^{Monge}(\mu, \cdot)&:= \left\{ \gamma=(id, T)_*\mu\colon\, T\colon\R^N\to\R^N \mbox{ is Borel, adapted and $\mu$-a.s.\ invertible} \right \},\label{def Monge bc}
\end{align}
where by ``$T$ is $\mu$-a.s.\ invertible'' we mean the existence of a Borel adapted map $R:\R^N\to\R^N $ such that 
$$R\circ T =id \,\, (\mu-a.s.)\,\,\,\,\mbox{ and }\,\,\, T\circ R =id\,\, (T_*\mu-a.s.).$$
We recall that $T=(T^1,\dots,T^N)$ is called adapted if $T^i(x_1,\dots,x_N)=T^i(x_1,\dots,x_i)$ for each $i$.
Mappings having the properties specified in \eqref{def Monge bc} have been called ``isomorphism of filtered probability spaces'' in the literature. Now our main result in this part. We use the notation \emph{ext} and \emph{conv} to denote the extreme points and the convex hull of a set.

%\begin{remark}
%Clearly $\Pi_{bc}^{Monge}(\mu, \cdot)\subset \Pi_{bc}(\mu, \cdot)$ so as in the above proof $\Pi_{bc}^{Monge}(\mu, \cdot)\subset \text{ext}\,\,\Pi_{bc}(\mu, \cdot)$.
%We conjecture that in fact   $\Pi_{bc}^{Monge}(\mu, \cdot) = \text{ext}\,\,\Pi_{bc}(\mu, \cdot)$.
%\end{remark}

\begin{theorem}\label{thm extreme}
We have
$$\Pi_{c}^{Monge}(\mu, \cdot) = \text{ext}\,\,\Pi_{c}(\mu, \cdot),$$
and
$$ \Pi_{bc}^{Monge}(\mu, \cdot) \subset \Pi_{c}^{Monge}(\mu, \cdot) \bigcap \Pi_{bc}(\mu,\cdot) = \text{ext}\,\,\Pi_{bc}(\mu, \cdot).$$
%$\Pi_{c}^{Monge}(\mu, \cdot)$ is precisely the set of extreme points of  $\Pi_{c}(\mu, \cdot)$, and $\Pi_{bc}^{Monge}(\mu, \cdot)$ is precisely the set of extreme points of $\Pi_{bc}(\mu, \cdot)$.
\end{theorem}
\begin{proof}
It is clear that $\Pi_{c}^{Monge}(\mu, \cdot)\subset \Pi_{c}(\mu, \cdot)$. From this one sees that $\Pi_{c}^{Monge}(\mu, \cdot)\subset \text{ext}\,\,\Pi_{c}(\mu, \cdot)$, since a coupling supported on the graph of a function cannot arise as the combination of two couplings without this property. Similarly, we have $\Pi_{bc}^{Monge}(\mu, \cdot)\subset \Pi_{c}^{Monge}(\mu, \cdot) \bigcap \Pi_{bc}(\mu,\cdot)\subset \text{ext}\,\,\Pi_{bc}(\mu, \cdot)$.

We first prove that $\Pi_{c}^{Monge}(\mu, \cdot)\supset  \text{ext}\,\, \Pi_{c}(\mu, \cdot)$. It is easy to see from \cite[Proposition 2.4]{BBLZ}, especially part $4.$ therein, that $\gamma \in \Pi_{c}(\mu, \cdot)$ is equivalent to
\[\textstyle\int F \mathrm d\gamma =0\]
for all $F$ of either of the two following forms:
\begin{itemize}
\item[(i)] $F= \phi(x_1,\dots,x_N)-\int \phi(\bar{x}_{1},\dots,\bar{x}_N)\mu(\mathrm d \bar{x}_{1},\dots,\mathrm d\bar{x}_N)$, with $\phi=:\phi^x$ bounded measurable,
\item[(ii)] $F=\sum_{t<N} H_t^y[M^x_{t+1}-M^x_t]$, with $H^y$ a bounded continuous process adapted to the $y$-variables,  $M^x$ a bounded martingale adapted to the $x$-variables (namely $H_t^x=H_t(x_1,\dots,x_t)$ and $M^y_t=M_t(y_1,\dots,y_t)$.
\end{itemize}
We consider now the vector space $V$ generated by the constant $1$ and the functions on $\R^N\times \R^N$ of the form $(i)$ and $(ii)$. Explicitly, we have
$$\textstyle V=\left\{c+\phi^x + \sum_t H^y_t[M^x_{t+1}-M^x_t] :\,\, \text{with }\phi^x,H^y,M^x \text{ as described in }(i)-(ii), c\in\R \right\}.$$
By Douglas Theorem \cite[Ch. V,(4.4)]{ReYo99}, we have that $\gamma\in \text{ext}\,\, \Pi_{c}(\mu, \cdot)$ is equivalent to $V$ being dense in $L^1(\gamma)$. We take $\gamma$ such an extreme point, an arbitrary $i\in\{1,\dots,N\}$ and $h$ Borel bounded function. We will show that $h(y_i)= \int h(\bar{y}_i)\gamma^{x_1,\dots,x_i}(\mathrm d \bar{y}_i)$ holds $\gamma$-a.s. This would immediately imply the existence of measurable functions $T^i$ such that for all $i: y_i=T^i(x_1,\dots,x_i)$ holds $\gamma$-a.s.. This then implies $\gamma\in\Pi_{c}^{Monge}(\mu, \cdot)$. 

We start observing that 
$$\textstyle \int h(y_i)\phi^x \mathrm d \gamma = \int  \left[ h(\bar{y}_i) \gamma^{x_1,\dots,x_N}(\mathrm d \bar{y}_i)\right ]  \phi^x \mathrm d \mu = \int  \left[ h(\bar{y}_i) \gamma^{x_1,\dots,x_i}(\mathrm d \bar{y}_i)\right ]  \phi^x \mathrm d \mu , $$
since by causality $x_{i+1},\dots,x_N$ are $\gamma$-independent of $y_i$ given $x_1,\dots,x_i$. Now we prove the desired $$\textstyle h(y_i)= \int h(\bar{y}_i)\gamma^{x_1,\dots,x_i}(\mathrm d \bar{y}_i),$$
which we do by induction in $i$. For $i=1$, we get
$$ \textstyle\int h(y_1)H_t^y[M^x_{t+1}-M^x_t]\mathrm d \gamma =  \int \left [ h(\bar{y}_1) \gamma^{x_1}(\mathrm d \bar{y}_i)\right ]H_t^y[M^x_{t+1}-M^x_t]\mathrm d \gamma ,$$
since by the same conditional independence argument both sides are equal to $0$ (indeed, $M^x$ must be by causality a martingale in the filtration of the $x$ and $y$ variables). It follows that $$\textstyle \int \left\{h(y_1) - \int h(\bar{y}_1) \gamma^{x_1}(\mathrm d \bar{y}_i)\right\} v(x_1,\dots,x_N,y_1,\dots,y_N)\, \mathrm d \gamma = 0, \,\,\forall v\in V.$$
Since $V$ is dense in $L^1(\gamma)$, we obtain the claim for $i=1$. Now let us suppose this has been proved for all indices $i\leq j$. In order to establish the result for $j+1$, the key is to prove that 
$$ \textstyle\int h(y_{j+1})H_t^y[M^x_{t+1}-M^x_t]\mathrm d \gamma =  \int \left [ h(\bar{y}_{j+1}) \gamma^{x_1,\dots,x_{j+1}}(\mathrm d \bar{y}_{j+1})\right ]H_t^y[M^x_{t+1}-M^x_t]\mathrm d \gamma .$$
But this is true by the same argument as above if $t>j$ (one verifies that both sides are equal to $0$). In case $t\leq j$, by induction we have that $H_t^y=\tilde{H}_t(x_1,\dots,x_t)$ $\gamma$-a.s.\ so we obtain that the l.h.s.\ is equal to $\int \left [ h(\bar{y}_{j+1}) \gamma^{x_1,\dots,x_{N}}(\mathrm d \bar{y}_{j+1})\right ]H_t^y[M^x_{t+1}-M^x_t]\mathrm d \gamma$, and this is the equal to the r.h.s.\ by causality.

We now prove $\Pi_{c}^{Monge}(\mu, \cdot)\supset  \text{ext}\,\, \Pi_{bc}(\mu, \cdot)$. Here $V$ must be replaced by
$$\textstyle \tilde{V}=\left\{\,\,c+\phi^x + \sum_t H^y_t[M^x_{t+1}-M^x_t] +  \sum_t G^x_t[N^y_{t+1}-N^y_t] \,\,\right\},$$
with the obvious extension of the notation used so far. By essentially the same arguments as above one obtains that for $\gamma\in \text{ext}\,\, \Pi_{bc}(\mu, \cdot)$ we have $y_t= T^t(x_1,\dots,x_t)$ $\gamma$-a.s. Indeed the only thing that need be observed, is that under $\gamma$ any martingale with respect to the $y$-filtration remains a martingale if we adjoin the $x$-filtration. This shows $\Pi_{c}^{Monge}(\mu, \cdot)\supset  \text{ext}\,\, \Pi_{bc}(\mu, \cdot)$ and hence  $\Pi_{c}^{Monge}(\mu, \cdot) \bigcap \Pi_{bc}(\mu,\cdot) = \text{ext}\,\,\Pi_{bc}(\mu, \cdot)$. 
\end{proof}

%\medskip 

\begin{remark}
The inclusion in Theorem \ref{thm extreme} is strict unless $\mu$ is concentrated in a single point. Indeed, for $a\in\R^N$ denote $T^a(x):=a$ and observe that $\gamma_a:=(id,T^a)_*\mu \in \Pi_{c}^{Monge}(\mu, \cdot) \bigcap \Pi_{bc}(\mu,\cdot) $. However, $\gamma_a \in \Pi_{bc}^{Monge}(\mu, \cdot)$ if and only if $\mu$ is concentrated in a point. 
\end{remark}

We now give an alternative argument for $\text{ext}\,\,\Pi_{c}(\mu, \cdot)\subset \Pi_{c}^{Monge}(\mu, \cdot)$. We think that this method, despite being more involved, it is more hands-on, and so it may be more suitable for future extensions of our results. We thank Daniel Lacker for pointing out the relevance  of this question, and for providing us with the reference \cite[IV-43(p.110)]{DeMeA} that we employ here:

%\begin{theorem}\label{thm extreme}
%We have
%$$\Pi_{c}^{Monge}(\mu, \cdot) = \text{ext}\,\,\Pi_{c}(\mu, \cdot).$$%\,\,\,\,\text{ and }\,\,\,\, \Pi_{bc}^{Monge}(\mu, \cdot) = \text{ext}\,\,\Pi_{bc}(\mu, \cdot).$$
%%$\Pi_{c}^{Monge}(\mu, \cdot)$ is precisely the set of extreme points of  $\Pi_{c}(\mu, \cdot)$, and $\Pi_{bc}^{Monge}(\mu, \cdot)$ is precisely the set of extreme points of $\Pi_{bc}(\mu, \cdot)$.
%\end{theorem}
\begin{proof}[Alternative Proof]
%It is clear that $\Pi_{c}^{Monge}(\mu, \cdot)\subset \Pi_{c}(\mu, \cdot)$. From this is direct to see that $\Pi_{c}^{Monge}(\mu, \cdot)\subset \text{ext}\,\,\Pi_{c}(\mu, \cdot)$, since a coupling supported on the graph of a function cannot arise as the combination of two couplings without this property. 

We show that if $\gamma$ is {not} in $\Pi_{c}^{Monge}(\mu, \cdot)$, then it can neither be an extreme point of $\Pi_{c}(\mu, \cdot)$. The argument is inspired by \cite[IV-43(p.110)]{DeMeA}. We let $$\tau=\min\{t\in\{1,\dots,N\}\,:\,\#\, \text{supp} \,\gamma^{x_1,\dots,x_t}(\mathrm dy_t)>1\},$$
with the convention $\min\,\emptyset = N+1$. Observe that $\tau$ is an $x$-stopping time in the sense that $\tau\leq t$ is measurable w.r.t.\ $\{x_1,\dots,x_t\}$ and that there exists Borel functions $F^t:\R^t\to \R$ such that $y_s = F^s(x_1,\dots,x_s)$ $\gamma$-a.s.\ on $s<\tau$. We now introduce a random variable $j$ by selecting $$j \in \text{int conv supp}\, \gamma^{x_1,\dots,x_t}(\mathrm dy_y)\,\,\,\, \text{ if }\,\,\,\, \tau=t\leq N, $$
in a measurable way (when $\tau=N+1$ we may simply set $j=0$ say). This is done with usual measurable selection arguments\footnote{Here an explicit construction: Assuming $\gamma^{x_1,\dots,x_N}(\mathrm dy_t)$ has a.s.\ finite mean for all $t$, then the process $z_t:=\text{mean}\gamma^{x_1,\dots,x_N}(\mathrm dy_t)$ is by causality adapted to the filtration of the $x$'s and defining $j:=z_\tau$ would suffice.} and more precisely, we have that $j=j(\tau,x_1,\dots,x_\tau)$. Letting 
$$J:=(-\infty,j),$$
the key point here is that
$$\gamma^{x_1,\dots,x_t}(y_t\in J)>0\,\, \text{ and }\,\, \gamma^{x_1,\dots,x_t}(y_t\notin J)>0, \,\,\,\text{$\gamma$-a.s.\ on }\{\tau=t\leq N\}.$$
Using this fact, we will define $\pi,\tilde{\pi}\in \Pi_{c}(\mu, \cdot)$ such that $\gamma = \lambda \pi+(1-\lambda)\tilde{\pi}$ where $\pi\neq \tilde{\pi}$ for $0< \lambda$ small enough. This will prove that $\gamma$ is not an extreme point. We fix $0<\lambda<1$ and for $t\leq N$ we first introduce
$$
\gamma_t^{x_1\dots,x_N}(\mathrm dy_t,\dots,\mathrm dy_N):= \left \{
\begin{array}{lll}
\frac{\gamma^{x_1,\dots,x_N}(\mathrm dy_t,\dots,\mathrm dy_N)  {\bf 1}_J(y_t)}{\gamma^{x_1,\dots,x_t}(y_t\in J)}&\text{if}\,\, \gamma^{x_1,\dots,x_t}(y_t\in J)>\lambda, \\
\gamma^{x_1,\dots,x_N}(\mathrm dy_t,\dots,\mathrm dy_N) & \text{otherwise}.
\end{array}
\right .
$$
Now we define $\pi$ via
\begin{align*}\textstyle
\pi(\mathrm dx_1,\dots,\mathrm dx_N,\mathrm dy_1,\dots,\mathrm d y_N):= &\mu(\mathrm dx_1,\dots,\mathrm dx_N) \Bigl \{{\bf 1}_{\tau=N+1} \gamma^{x_1,\dots,x_N}(dy_1,\dots,\mathrm dy_N) +{\bf 1}_{\tau=1}\gamma_1^{x_1,\dots,x_N}(dy_1,\dots,\mathrm dy_N) \Bigr . \\ &\textstyle  \Bigl . +\sum_{t=2}^N {\bf 1}_{\tau=t}\gamma^{x_1,\dots,x_{t-1}}(dy_1,\dots,\mathrm dy_{t-1})\gamma_t^{x_1,\dots,x_N}(dy_t,\dots,\mathrm dy_N) \Bigr \} .
\end{align*}
Proving that $\pi\in \Pi_{c}(\mu, \cdot)$ is tedious but straightforward. We illustrate it by showing that $\pi(\mathrm dx_1) = \mu(\mathrm dx_1)$ and $\pi^{x_1,y_1}(\mathrm dx_2)=\mu^{x_1}(\mathrm d x_2)$. Iterating this and using Proposition \ref{prop:equiv charact} (more specifically the analogue result in \cite{BBLZ}) we get $\pi\in \Pi_{c}(\mu, \cdot)$. We start by the computation
\begin{align*}\textstyle \int A(x_2)H(x_1,y_1)\mathrm d\pi =& \textstyle \int \left ( \int A(x_2) \mu^{x_1}(\mathrm d x_2)\right )\Bigl [ H(x_1,y_1){\bf 1}_{\tau\geq 2}\mathrm d\gamma + H(x_1,y_1){\bf 1}_{\tau=1,\gamma^{x_1}(y_1\in J)\leq \lambda}\mathrm d\gamma \Bigr . \\
& \textstyle\Bigl . + \int_J H(x_1,y_1)\frac{\gamma^{x_1}(\mathrm dy_1)}{\gamma^{x_1}(y_1\in J)}{\bf 1}_{\tau=1,\gamma^{x_1}(y_1\in J)> \lambda} \mu(\mathrm dx_1)\Bigr ],
\end{align*}
where we used $\gamma\in \Pi_{c}(\mu, \cdot)$ and that $\tau$ is a stopping time. Taking $A\equiv 1$ and $H=H(x_1)$ this gives $\pi(\mathrm dx_1) = \mu(\mathrm dx_1)$, since the $x$-marginal of $\gamma$ is $\mu$. On the other hand,  $\pi^{x_1,y_1}(\mathrm dx_2)=\mu^{x_1}(\mathrm d x_2)$ is apparent.

We now define $$ \textstyle\tilde{\pi}:=\frac{\gamma-\lambda \pi}{1-\lambda},$$
so that $\gamma = \lambda \pi+(1-\lambda)\tilde{\pi}$ and clearly the $x$-marginal of $\tilde{\pi}$ is $\mu$. On the other hand, under the assumption that $\gamma$ is not in $\Pi_{c}^{Monge}(\mu, \cdot)$, we know that $\tau\leq N$ with positive probability, and so $\gamma\neq \pi$ if $\lambda$ is small enough. To finish the proof it only remains to establish that $\tilde{\pi}$ is a probability measure which is causal.

Since on $\{\tau=t\leq N\}$ we have $(y_1,\dots,y_{t-1})=(F^1(x_1),\dots,F^{t-1}(x_1,\dots,x_{t-1}))$, by the tower property and causality it follows $$\gamma^{x_1,\dots,x_N}(\mathrm dy_1,\dots, \mathrm dy_N)=\gamma^{x_1,\dots,x_{t-1}}(\mathrm dy_1,\dots, \mathrm dy_{t-1})\gamma^{x_1,\dots,x_N}(\mathrm dy_t,\dots, \mathrm dy_N),\,\,\,\text{ on $\{\tau=t\leq N\}$}.$$
Using this we may decompose $\tilde{\pi}$ in the same fashion as in the definition of $\pi$, but where we replace $\gamma_t$ by
$$
\tilde{\gamma}_t^{x_1\dots,x_N}(\mathrm dy_t,\dots,\mathrm dy_N):= \left \{
\begin{array}{lll}
(1-\lambda)^{-1}\gamma^{x_1,\dots,x_N}(\mathrm dy_t,\dots,\mathrm dy_N)  \left ( 1 - \frac{ \lambda{\bf 1}_J(y_t)}{\gamma^{x_1,\dots,x_t}(y_t\in J)}\right )&\text{if}\,\, \gamma^{x_1,\dots,x_t}(y_t\in J)>\lambda, \\
\gamma^{x_1,\dots,x_N}(\mathrm dy_t,\dots,\mathrm dy_N) & \text{otherwise}.
\end{array}
\right .
$$
From this it is apparent that $\tilde{\pi}$ defines a non-negative measure and is actually a probability measure. The causality property of $\tilde{\pi}$ can be proved as for $\pi$.
\end{proof}

\section{The Knothe--Rosenblatt distance}\label{sec:KnotheRosenblatt}

Perhaps the most eminent of bicausal plans (i.e., those participating in the determination of the nested distance) is the so-called increasing Knothe--Rosenblatt rearrangement, which is also known as quantile transform or increasing triangular transform in the literature. See, e.g., \cite{BBLZ} and references therein. Let us introduce some useful notation first. By $F_{\eta}(\cdot)$ we denote the distribution function of a probability measure $\eta$ on the line (we denote $F_{\nu_1}$ the distribution of $p^1_*\nu$) and by  $F^{-1}_{\eta}(\cdot)$ its left-continuous generalized inverse, i.e., $F^{-1}_{\eta}(u) = \inf\left\lbrace y: F_{\eta} (y) \geq u \right\rbrace $. The (increasing $N$-dimensional) \emph{Knothe--Rosenblatt} rearrangement of $\mu$ and $\nu$ is defined as the law $\pi$ of the random vector $(X_1^*,\dots,Y_N^*, X_1^*,\dots,Y_N^*)$, where
\begin{align}\label{quantile transforms}\textstyle
& X_1^* = F_{\mu_1}^{-1} (U_1),\hspace{46pt}  Y_1^* = F_{\nu_1}^{-1} (U_1), \hspace{31pt}\mbox{ and inductively }\\ \nonumber
& X_t^* = F_{\mu^{X_1^*,\dots,X_{t-1}^*}}^{-1} (U_t),\quad Y_t^* = F_{\nu^{Y_1^*,\dots,Y_{t-1}^*}}^{-1} (U_t),\, \text{for } t=2,\dots,N,
\end{align}
for $U_1,\dots,U_N$ independent and standard uniformly distributed random variables. Additionally, if $\mu$-a.s.\ all the conditional distributions of $\mu$ are atomless (e.g., if $\mu$ has a density), then this rearrangement is induced by the map%\footnote{Also called \textit{increasing triangular transform}}
\[(x_1,\dots,x_N)\mapsto T(x_1,\dots,x_N):=(T^1(x_1),T^2(x_2|x_1),\dots,T^N(x_N|\,x_1, \dots. x_{N-1})),\]
where 
\begin{align}\textstyle
T^1(x_1)&:= \, F_{\nu_1}^{-1}\circ F_{\mu_1}(x_1), \nonumber\\
T^t(x_t|\, x_1,\dots,x_{t-1})& :=\, F_{\nu^{T^1(x_1),\dots,T^{t-1}(x_{t-1}|\,x_1,\dots,x_{t-2})}}^{-1}\circ F_{\mu^{x_1,\dots,x_{t-1}}}(x_t),\,\,\,\, t\geq 2. \label{KRdef}
\end{align}

In this section we reserve the letters $\pi$ and $T$ for this rearrangement (map) and omit its dependence on $(\mu,\nu)$, which is clear from the context. Let us now define a functional on $\P^p(\R^N)\times\P^p(\R^N)$ which we compare with the nested distance in Section~\ref{sec:KRvsNested} below.

\begin{definition}
The \emph{Knothe--Rosenblatt distance} of order $p$ (in short KR distance) is defined by
\begin{equation}\label{KR distance}
d_p^{KR}(\mu,\nu):=\left(\iint d^p \mathrm  d\pi \right )^{1/p},
\end{equation}
where $\pi$ is the Knothe--Rosenblatt rearrangement of $\mu$ and $\nu$.
\end{definition}

\begin{lemma}
\label{KR metric}
The KR distance $d_p^{KR}$ is a metric on $\P^p(\R^N)$ for any $1\leq p<\infty$.
\end{lemma}

%\comment{JB: Triangle inequality is trivial for $p=1$, and otherwise should follow by Minkowski inequality.}
\begin{proof}
Since both $d^p$ and the Knothe--Rosenblatt rearrangement are symmetric, we see that  $d_p^{KR}$ is symmetric. Obviously this distance is non-negative and vanishes exactly when $\mu=\nu$, since one dimensional distributions and conditional distributions fully encode a measure. Finally observe that $d_p^{KR}(\mu,\nu)=\left( E[d(X^*,Y^*)^p]\right)^{1/p}$, where %$\|\cdot\|_p$ is the $p$-norm in $\R^N$ and 
$X^*$, $Y^*$ are as in \eqref{quantile transforms}. Given $\eta\in\P^p(\R^N)$ and constructing $Z^*$ as in~\eqref{quantile transforms} so that it is distributed like $\eta$,  the triangle inequality follows from $\left( E[d(X^*,Y^*)^p]\right)^{1/p}\leq \left( E[d(X^*,Z^*)^p]\right)^{1/p}+\left( E[d(Z^*,Y^*)^p]\right)^{1/p}$.
\end{proof}

For its simplicity and the fact that the Knothe--Rosenblatt rearrangement has already found startling applications all over mathematics, one is tempted to explore the connection between the KR and the nested distances. {Furthermore, by the optimality properties of the increasing rearrangement on the line, it is clear that $d_p^{nd}= d_p^{KR}$ for $N=1$ and e.g.\ $d(x,y)=|x-y|$, in which case both metrics coincide with the usual $p$-Wasserstein metric on $\P(\R)$. Can it be that $d_p^{nd}$ is always comparable/similar, if not equal, to $d_p^{KR}$?. We examine this now.}

\subsection{Relationship with the nested distance}
\label{sec:KRvsNested}

By definition we have that 
\begin{align}
	d_p^{\nd}(\mu,\nu)\leq d_p^{KR}(\mu,\nu). \label{trivial}
\end{align} 

We ask,
\begin{enumerate}
	\item\label{enu:1} Is there a constant $C>0$ such that $d_p^{KR}(\mu,\nu)\leq C d_p^{\nd}(\mu,\nu)$ holds, i.e., are the two metrics strongly equivalent?
	
	\item\label{enu:2} Is it the case that $d_p^{\nd}(\mu_n,\mu)\to 0$ implies $d_p^{KR}(\mu_n,\mu)\to 0$, i.e., are the two metrics topologically equivalent? 
\end{enumerate}

{
From now we specialize to the case $N=2$ but the situation is the same for all $N>1$.} The next counterexample shows that the answer to question~\ref{enu:1} is negative:

\begin{example}
	Let $\mu_n:=1/2[\delta_{(1/n,n/2)}+\delta_{(-1/n,-n/2)}]$ and $\nu_n:=1/2[\delta_{(1/n,-n/2)}+\delta_{(-1/n,n/2)}]$ and take $\underline{d}$ the usual distance in $\R$. Observe that any transport plan is bicausal in this setting. Thus, we bound $d_p^{\nd}(\mu_n,\nu_n)$ from above by the value of the decreasing Knotte-Rosenblatt map (associating $\{1/n,n/2\}$ to  $\{-1/n,n/2\}$ and $\{-1/n,-n/2\}$ to $\{1/n,-n/2\}$)
	\[d_p^{\nd}(\mu_n,\nu_n)\leq 2/n,\]
	whereas for the usual (increasing) Knotte-Rosenblatt map we get
	\[d_p^{KR}(\mu_n,\nu_n)= n,\]
and letting $n\to\infty$ we see that~\ref{enu:1} cannot hold.
\end{example}

Actually, even~\ref{enu:2} fails, as the following counterexample demonstrates:

\begin{example}
	Consider the measure $P:=\lambda\otimes\lambda$ on $\Omega:=[0,1]^2$ and the random variables (i.e., the two-stage stochastic processes)
	\[Z_n:=\begin{cases}
		(0,\,u_2)   & \text{if } u_1\le\frac{1}{2},\\
		(\frac{1}{n},\ 1+u_2) & \text{if } u_1> \frac{1}{2}
	\end{cases}\]
	for $n=1,2,\dots $ and
	\[Z_\infty:=\begin{cases}
	(0,\,u_2)   & \text{if } u_1\le\frac{1}{2},\\
	(0,\ 1+u_2) & \text{if } u_1> \frac{1}{2}
	\end{cases}\]
	together with their image measures \[\mu^{(n)}:=P\circ Z_n^{-1}\text{ and }\mu:=P\circ Z_\infty^{-1}.\] It is evident that we have convergence in nested distance, $d_p^{\nd}(\mu_n,\mu)\to0$ as $n\to\infty$, as we may employ the conditional transport maps \[T(x|\,U_1)=x\] on the second stage (i.e., $T(u_2|\,u_1\ge\frac{1}{2})=u_2$  and $T(1-u_2|\,u_1<\frac{1}{2})=1-u_2$).
	%Note that they depend on $\omega\in\Omega$, via $U_1$.

	To compute the Knothe--Rosenblatt distance we associate $\mu^{(n)}$ with $X^*$ in~\eqref{quantile transforms} and $\mu$ with $Y^*$. Then
	\[X_1^*=\begin{cases}
		0&\text{if }U_1\le\frac{1}{2},\\
		\frac{1}{n}&\text{if }U_1>\frac{1}{2}
	\end{cases}\quad\text{ and }\quad Y_1^*=0,\]
	while
	\[X_2^*=\begin{cases}
		U_2&\text{if }U_1\le\frac{1}{2},\\
		1+U_2&\text{if }U_1>\frac{1}{2},
	\end{cases}\quad\text{but}\quad
	Y_2^*=2\,U_2.\]
	The second components are entirely different ($X_2^*$ depends on $U_1$, while $Y^*_2$ does not) and hence the Knothe--Rosenblatt distance does \emph{not} tend to $0$, $d_p^{KR}(\mu_n,\mu)\not\to0$, as $n\to\infty$. 
\end{example}

	We conclude that~\ref{enu:2} does \emph{not} hold true and the topologies induced by the Knothe--Rosenblatt and the nested distances differ. More generally, we expect that there is no distinguished bicausal transport that can generate a topology compatible with the nested distance for $N>1$. 

%The following example does not contradict (ii), but may be illustrative in its own right:
%
%\begin{example}
%Let $\mu^N:=1/2[\delta_{1/N,N}+\delta_{0,0}]$ and $\nu^N:=1/2[\delta_{1/N,0}+\delta_{0,N}]$. By similar arguments we get
%%
%$$  d_p^{\nd}(\mu^N,\nu^N)\leq 1/N  ,$$
%%
%whereas 
%%
%$$  d_p^{KR}(\mu^N,\nu^N)= N  ,$$
%%
%so $d_p^{\nd}(\mu^N,\nu^N)\to 0$ while $d_p^{KR}(\mu^N,\nu^N)\to \infty$. %Crucially neither $\mu^N$ nor $\nu^N$ converge. One can modify this example so that the support of the marginals do not ``diverge'', but it seems that in order to obtain $d_p^{\nd}\to 0$ and yet $d_p^{KR}$ bounded away from zero, one still needs that either marginal is not converging in nested distance
%\end{example}

%Even though we do not have a formal proof (or counterexample) of~\ref{enu:2}, certain heuristics suggest that the answer to this question is positive. 

\section{Summary}\label{sec:Summary}
%This article addresses the nested (or process) distance, which is a distance for stochastic processes. 
%The nested distance generalizes the Wasserstein distance by incorporating the filtration in addition. It is of particular importance in stochastic optimization, as multistage stochastic optimization problems are continuous with respect to it. Furthermore, this distance is a sharp and tight upper bound for the difference between stochastic programs based on different probability measures.

In this article we investigated fundamental topological properties of the nested distance. In contrast to classical Wasserstein distances, for example, its topology cannot be characterized via integration on test functions, so that complete stochastic programs appear as the natural distinguishing element of the topology induced by the nested distance. %We provide equivalent characterizations of  the weak nested topology.
The nested distance is also not complete, again in contrast to the classical Wasserstein distance. 

We obtained two main results. First, we demonstrated that the metric completion of the nested distance is the space of nested distributions with their classical Wasserstein metric, as introduced in~\cite{Pflug2009}. This provides a connection between two hitherto unrelated mathematical objects. Second, we established that the topology generated by the nested distance is Polish, which we hope opens the way to future applications. Along these lines, we have started in this article the study of extreme points of sets of measures relevant for stochastic optimization.

We finally introduced the Knothe--Rosenblatt distance between processes, which likewise takes into account the filtration structure. Its appeal lies in its apparent simplicity and the eminence of the Knothe--Rosenblatt rearrangement, as well as in the fact that in the one-period setting this distance and the nested distance coincide. In this regard, we demonstrate that when more than one period is considered, the topology generated by the Knothe--Rosenblatt distance is strictly stronger than the nested distance topology.% for dimension different than $1$.%, so the Knothe--Rosenblatt distance cannot replace the nested distance.

\bibliographystyle{plain}

\bibliography{biblio_causal_transport,LiteraturAlois,joint_biblio}

\end{document}